\documentclass[12pt, oneside, psamsfonts]{amsart}

\newif\ifPDF
\ifx\pdfoutput\undefined\PDFfalse
\else \ifnum \pdfoutput > 0 \PDFtrue
        \else \PDFfalse
        \fi
\fi

\usepackage[centertags]{amsmath}
\usepackage{amsfonts}
\usepackage{mathrsfs}
\usepackage{textcomp}
\usepackage{amssymb}
\usepackage{amsthm}
\usepackage{newlfont}
\usepackage[all]{xy}
\usepackage{stmaryrd}

\ifPDF
  \usepackage[pdftex]{color, graphicx}
  \usepackage[pdftex, bookmarks, colorlinks]{hyperref}
  \hypersetup{colorlinks=false}

\else
  \usepackage{color}
  \usepackage[dvips]{graphicx}
  \usepackage[dvips]{hyperref}
\fi

\usepackage[scale=0.8]{geometry}

\newtheorem{thm}{Theorem}[section]
\newtheorem{cor}[thm]{Corollary}
\newtheorem{lem}[thm]{Lemma}
\newtheorem{prop}[thm]{Proposition}
\theoremstyle{definition}
\newtheorem{defn}[thm]{Definition}
\theoremstyle{remark}
\newtheorem{rem}[thm]{Remark}
\newtheorem{example}[thm]{Example}
\numberwithin{equation}{section}

\newcommand{\norm}[1]{\left\Vert#1\right\Vert}
\newcommand{\abs}[1]{\left\vert#1\right\vert}

\newcommand{\Real}{\mathbb R}
\newcommand{\Int}{\mathbb Z}
\newcommand{\Comp}{\mathbb C}

\newcommand{\eps}{\varepsilon}

\newcommand{\Kzero}{\mathrm{K}_0}
\newcommand{\Kone}{\mathrm{K}_1}
\newcommand{\tr}{\mathrm{T}}

\newcommand{\aff}{\mathrm{Aff}}

\begin{document}

\title[A Generalized construction and the comparison radius function]{Remarks on Villadsen algebras, II: A generalized construction and the comparison radius function}

\author{George A. Elliott}
\address{Department of Mathematics, University of Toronto, Toronto, Ontario, Canada~\ M5S 2E4}
\email{elliott@math.toronto.edu}

\author{Zhuang Niu}
\address{Department of Mathematics and Statistics, University of Wyoming, Laramie, Wyoming 82071, USA.}
\email{zhuangniu@icloud.com}

%\thanks{}
%\keywords{}
\date{\today}
%\dedicatory{}
%\commby{}

%------------------------------------------abstract--------------------------------------

\begin{abstract}

The authors' recent classification of Jesper
Villadsen's remarkable generalization (based on a
self-reproducing seed space) of Glimm's infinite tensor
product (UHF) C*-algebras, by means of the Cuntz semigroup
(in the case of a fixed, well-behaved, seed space),
is extended to the analogous generalization of Bratteli's
approximately finite-dimensional (AF) C*-algebras. Some
progress is made in the direction of distinguishing between
algebras based on different seed spaces.

%Villadsen's construction is generalized for arbitrary Bratteli diagrams. Similar to the classification of the UHF case, it is shown that if the seed space is a finite dimensional K-contractible solid space, then these AF-Villadsen algebras are classified by the Cuntz semigroup. In fact, they are classified by their $\Kzero$-group together with comparison radius functions.

%The more general version of comparison radius functions are also considered for the UHF-Villadsen algebras. It turns out can be used to distinguish certain UHF-Villadsen algebras with different (contractible) seed spaces, but same $\Kzero$-group and radius of comparison.
\end{abstract}

\maketitle

\section{Introduction}

 In \cite{ELN-Vill}, a beginning was made on classifying
simple C*-algebras beyond what might now be called the classical
classifiable class (often just called ``classifiable") in which
K-theory and traces suffice (see, for instance, in the unital finite case under consideration, \cite{EGLN-DR}), by using new information
contained in the Cuntz semigroup. (In particular, what was
used in \cite{ELN-Vill} was the Toms radius of comparison.)

  In \cite{ELN-Vill}, what might be called Villadsen algebras of the
first kind (introduced in \cite{Vill-perf}, and quite different from the
algebras studied later by Villadsen in \cite{Vill-sr}), or UHF-Villadsen
algebras (as they reduce to Glimm's infinite tensor product
algebras when the seed space is a single point),
with a fixed well-behaved seed space (for instance a cube),
were classified. In the present paper, this is extended
to the analogous class of AF-Villadsen algebras.

Examples of this class of algebras were constructed
by Hirshberg and Phillips in \cite{HP-Villadsen}, to show that
the particular Cuntz semigroup information used in \cite{ELN-Vill} 
(the radius of comparison) was no longer sufficient.
In the present paper, we consider more detailed information,
replacing the radius of comparison, a single number (or
as suggested in \cite{HP-Villadsen}---see also \cite{AGP-EDS} and \cite{Asadi-Vasfi:2025aa}---, a single number for each projection),
by a function on the tracial simplex which we call the
comparison radius function (with supremum the Toms invariant---see Corollary \ref{rc-V} and Theorem \ref{rc-V-general}):
\begin{thm}[Theorem \ref{classification-AF}]
Let $X$ be a K-contractible solid space such that $ 0 < \mathrm{dim}(X) < \infty$, and let $A(X, G, \mathcal E)$ and $B(X, H, \mathcal F)$ be AF-Villadsen algebras with seed space $X$ with rapid dimension growth (see \eqref{rapid-growth-cond}), where $G$ and $H$ are Bratteli diagrams and $\mathcal E$ and $\mathcal F$ are point evaluation sets. Then $A \cong B$ if, and only if, $(\mathrm{Cu}(A), [1_A]) \cong (\mathrm{Cu}(B), [1_B])$. Indeed, $A \cong B$ if, and only if, $$((\Kzero(A), \Kzero^+(A), [1_A]_0), r_\infty^{(0)}{(A)}) \cong ((\Kzero(B), \Kzero^+(B), [1_B]_0), r_\infty^{(0)}(B)),$$
where $r_\infty^{(0)}{(A)}$ and $r_\infty^{(0)}{(B)}$ are the comparison radius functions of $A$ and $B$ respectively.
\end{thm}

The comparison radius function $r_\infty^{(0)}$ considered in the theorem above is a continuous affine function on $\mathrm{S}_u(\Kzero(A))$. Note that, the theorem above implies that the trace simplex is determined by the $\Kzero$-group.

A more general comparison radius function, denoted by $r_\infty$, which is an upper semicontinuous affine function on $\mathrm{T}(A)$, is also constructed in the  UHF-Villadsen algebra case:

\begin{thm}[Theorem \ref{low-env-gn}]
Let $A$ be a UHF-Villadsen algebra with seed space a (finite) simplicial complex. There is a upper semicontinuous positive valued affine function $r_\infty$ on $\mathrm{T}^+(A)$ with  the following properties:
\begin{enumerate}

\item If $h \in \mathrm{Aff}(\mathrm{T}^+(A))$ (continuous affine functions, $0$ at $0$) and $r_\infty \leq h$, then, $h$ has the property that for any $a, b \in (A\otimes \mathcal K)^+$,
$$\mathrm{d}_\tau(a) + h(\tau) < \mathrm{d}_\tau(b),\ \tau \in \mathrm{T}^+(A) \quad  \Rightarrow \quad  a \precsim b. $$

\item If $h \in \mathrm{Aff}(\mathrm{T}^+(A))$ and $h(\tau_0) < r_\infty(\tau_0)$ for some $\tau_0 \in \mathrm{T}^+(A)$, then, there are $a, b \in (A\otimes \mathcal K)^+$ such that $$\mathrm{d}_\tau(a) + h(\tau) < \mathrm{d}_\tau(b),\quad \tau \in  \mathrm{T}^+(A),$$ but $a$ is not Cuntz-subequivalent to $b$.

%\item The function $r_\infty$ is upper semicontinuous. % i.e., for any $\eps>0$ and any $\tau_0 \in \tr^+(A)$, there is an open neighbourhood $U$ of $\tau$ such that $$r_\infty(\tau) < r_\infty(\tau_0) + \eps,\quad \tau \in U.$$

\end{enumerate}
%So, the upper semicontinuous function $r_\infty$ is the lower enveloping function of the set $G_A$ of continuous gap functions.
\end{thm}

The general comparison radius function $r_\infty$ in fact can be used to join together the classification theorems of \cite{ELN-Vill} for certain different seed spaces (and, by Corollary 7.9 of \cite{ELN-Vill}, their finite Cartesian powers):
\begin{thm}[Corollary \ref{joint-classification}]
Let $A = A(X^{(A)}, (n^{(A)}_i), (k^{(A)}_i))$ and $B = B(X^{(B)}, (n^{(B)}_i), (k^{(B)}_i))$  be UHF-Villadsen algebras with seed spaces $[0, 1]^2$ or $ [0, 1] \vee [0, 1]^2$. Then $A \cong B$ if, and only if, $(\mathrm{Cu}(A), [1_A]) \cong (\mathrm{Cu}(B), [1_B])$. Indeed, $A \cong B$ if, and only if, $$((\Kzero(A), \Kzero^+(A), [1_A]_0), r_\infty{(A)}) \cong ((\Kzero(B), \Kzero^+(B), [1_B]_0), r_\infty(B)).$$
\end{thm}

The general comparison radius function $r_\infty$  also can be used to show that the action of $\mathrm{Aut}(A)$ on the extreme points of $\tr(A)$, which is the Poulsen simplex (see \cite{ELN-Vill}),  is not transitive for certain UHF-Villadsen algebras $A$ (Corollary \ref{transitivity}).

The question of what the structure of the Cuntz semigroup for Villadsen algebras actually is is clearly of considerable interest. Since the algebras are of stable rank one (as proved by Villadsen), the recent result of Thiel, and Antoine, Perera, Robert, and Thiel, in \cite{Thiel-sr1} and \cite{APRT-sr1}, that the rank map splits, is very much pertinent.

\subsection*{Acknowledgements} The research of the first named author was supported by a Natural Sciences and Engineering Research Council of Canada (NSERC) Discovery Grant, and the research of the second named author was supported by a Simons Foundation grant (MP-TSM-00002606).

\section{AF-Villadsen algebras and the function $r_\infty^{(0)}$}\label{AF-construction}

\subsection{Growth along a Bratteli diagram}

Fix a metrizable compact space $X$ as the seed space. Consider an inductive sequence $(G_n, \phi_n)$, where $G_n = \Int^{s_n}$ with order unit $u_n = (u_{n, 1}, ..., u_{n, s_n})$. Consider the following inductive sequence of C*-algebras.

Set $$ \bigoplus_{j=1}^{s_n} \mathrm{M}_{u_{n, j}}(\mathrm{C}(X^{u_{n, j}})) = A_n.$$ For each $1\leq j \leq s_{n+1}$, note that the map $\phi_{n}: G_n \to G_{n+1}$ is induced by a multiplicity matrix $(m^{(n)}_{i, j})$, $1\leq i \leq s_n$, $1\leq j\leq s_{n+1}$. 
Then, for each $1 \leq j\leq s_{n+1}$, choose a partition
$$P_1 \sqcup \cdots \sqcup P_{s_n} = \{1, 2, ..., u_{n+1, j}\}$$
such that $$\abs{P_i} = m^{(n)}_{i, j}u_{n, i}.$$ Inside each $P_i$, choose another partition 
$$P_i = P_{i, 1} \sqcup \cdots \sqcup P_{i, m^{(n)}_{i, j}} $$
such that
$$\abs{P_{i, 1}} = \cdots = \abs{P_{i, m^{(n)}_{i, j}}}.$$
For each $k \in \{1, ..., m_{i, j}^{(n)}\}$, define $\pi_k$ to be the projection of $X^{u_{i+1, j}}$ onto the coordinate subset $P_{i, k}$. In this case, we also write $\pi_k \in P_i$.
%Note that, there is a permutation $\sigma$ of $\{1, 2, ..., u_{n+1, j} \}$ such that
%$$P_{i} = \{\sigma(\sum_{i'<i} m^{(n)}_{i', j}u_{n, i'}+1), ..., \sigma(\sum_{i'\leq i} m^{(n)}_{i', j}u_{n, i'})\}.$$
Then, define a map $A_n \to A_{n+1}$:
$$\phi_{n}: A_n \ni  (f_1, ..., f_{s_n}) \mapsto \bigoplus_{j=1}^{s_{n+1}} \bigoplus_{i=1}^{s_n} \bigoplus_{\pi_k\in P_i} f_i\circ \pi_k  \in A_{n+1}. $$
Let us call $P_i$ the supporting coordinates of $f_i$, and define the shape of $(\varphi_n)_j$ as $$((u_{n, 1}, m_{1, j}^{(n)}), ..., (u_{n, s_n}, m_{s_n, j}^{(n)})).$$

%Write 
%$$X^{u_{n+1, j}} = \underbrace{X^{u_{n, 1}} \times \cdots \times X^{u_{n, 1}}}_{m^{(n)}_{1, j}} \times \cdots \times \underbrace{X^{u_{n, s_n}} \times \cdots \times X^{u_{n, s_n}}}_{m^{(n)}_{s_n, j}},$$ and   define 
%$$\varphi_{n}: A_n \ni  (f_1, ..., f_{s_n}) \mapsto \bigoplus_{j=1}^{s_{n+1}} \mathrm{diag}\{ \underbrace{f_1 \circ \pi_1 ,..., f_1\circ \pi_{m^{(n)}_{1, j}}}_{m^{(n)}_{1, j}} ,...,  \underbrace{f_{s_n} \circ \pi_1 ,..., f_{s_n}\circ \pi_{m^{(n)}_{s_n, j}}}_{m^{(n)}_{s_n, j}} \} \in A_{n+1}. $$
%
%$$\varphi_{n}: A_n \ni  (f_1, ..., f_{s_n}) \mapsto \bigoplus_{j=1}^{s_{n+1}} \mathrm{diag}\{ \underbrace{f_1 \circ \pi^{(1)}_1 ,..., f_1\circ \pi^{(1)}_{m^{(n)}_{1, j}}}_{m^{(n)}_{1, j}} ,...,  \underbrace{f_{s_n} \circ \pi^{(s_n)}_1 ,..., f_{s_n}\circ \pi^{(s_n)}_{m^{(n)}_{s_n, j}}}_{m^{(n)}_{s_n, j}} \} \in A_{n+1}. $$ 

Denote the (non-simple) limit algebra by $A(X, (G_n, \phi_n), \mathcal P)$, where $\mathcal P$ denotes the choice of partitions.

\begin{rem}
The reason to include the partition is that, unlike the UHF case, even if we define $\phi_{n}$ and $\phi_{n+1}$ so that their partitions are standard, the partition associated to the composition $\phi_{n+1}\circ\phi_n$ is not standard. Therefore, there is no canonical choice for the supporting coordinates of the functions $f_i$, $i=1, ..., s_n$. However, we shall now show that, up to isomorphism,  the limit algebra $A(X, (G_n, \phi_n), \mathcal P)$ is independent of the choice of supporting coordinates. Therefore, we can omit $\mathcal P$, and just denote the algebra by $A(X, (G_n, \phi_n))$.
\end{rem}

\begin{prop}
In the setting above, one has $A(X, (G_n, \phi_n), \mathcal P) \cong B(X, (G_n, \phi_n), \mathcal Q)$ for any two partition systems $\mathcal P$ and $\mathcal Q$.
\end{prop}

\begin{proof}
It is enough to construct isomorphisms $\sigma_n: A_n \to B_n$, $n=1, 2,$..., such that 
\begin{equation}\label{comm-pert-cord}
\phi_n^{(B)} \circ \sigma_n = \sigma_{n+1} \circ \phi_n^{(A)},\quad n=1, 2, ... .
\end{equation}

Set $\sigma_1 = \mathrm{id}$. Assume that $\sigma_n$ is defined, and that it is induced by homeomorphisms $X^{u_{n, i}} \to X^{u_{n, i}}$, $i=1, ..., s_n$, by the coordinate permutations:
$$(x_1, ..., x_{u_{n, i}}) \mapsto (x_{\sigma_i(1)}, ..., x_{\sigma_i(u_{n, i})}),\quad i=1, ..., s_n. $$

A calculation shows that, up to a permutation,  for each $j=1, ..., s_{n+1}$,
\begin{eqnarray*}
&& (\phi_n^{(A)}(f_1, ..., f_{s_n}))_j(x_1, ..., x_{u_{n+1, j}}) \\
& = &  \mathrm{diag}\{\underbrace{f_1(x_1, ..., x_{u_{n, 1}}), ...,f_1(x_{(m_{1, 1}-1)u_{n, 1}+1}, ..., x_{m_{1, j} u_{n, s_n}})}_{m_{1, j}}, ..., \\
& & \quad\quad \underbrace{f_{s_n}(x_{u_{n+1, j}-m_{s_n, j}u_{n, s_n}+1}, ..., x_{u_{n+1, j}-(m_{s_n, j})u_{n, s_n}+u_{n, s_n}}), ...,f_{s_n}(x_{u_{n+1, j}-u_{n, s_n}+1}, ..., x_{u_{n+1, j}})}_{m_{s_n, j}}   \}.
\end{eqnarray*}
(As defined above, $(\phi_n^{(A)})_j$ has the shape $((u_{n, 1}, m^{(n)}_{1, j}), ..., (u_{n, s_n}, m^{(n)}_{s_n, j}))$.)

On the other hand, 
\begin{eqnarray*}
&& (\phi_{n}^{(B)}(\sigma_n(f_1(y_1, ..., y_{u_{n, 1}}), ..., f_{s_n}(y_1, ..., y_{u_{n, s_n}}))))_j(x_1, ..., x_{u_{n+1}, j}) \\
& = & (\phi_{n}^{(B)}(f_1(y_{\sigma_1(1)}, ..., y_{\sigma_1(u_{n, 1})}), ..., f_{s_n}(y_{\sigma_{s_n}(1)}, ..., y_{\sigma_{s_n}(u_{n, s_n})})))_j (x_1, ..., x_{u_{n+1}, j})  \\
& = &  \mathrm{diag}\{\underbrace{f_1(\sigma_1(x_1, ..., x_{u_{n, 1}})), ...,f_1(\sigma_1(x_{(m_{1, 1}-1)u_{n, 1}+1}, ..., x_{m_{1, j} u_{n, s_n}}))}_{m_{1, j}}, ..., \\
& & \quad\quad \underbrace{f_{s_n}(\sigma_{s_n}(x_{u_{n+1, j}-m_{s_n, j}u_{n, s_n}+1}, ..., x_{u_{n+1, j}-(m_{s_n, j})u_{n, s_n}+u_{n, s_n}})), ...,f_{s_n}(\sigma_{s_n}(x_{u_{n+1, j}-u_{n, s_n}+1}, ..., x_{u_{n+1, j}}))}_{m_{s_n, j}}   \}.
\end{eqnarray*}
So, $(\phi_{n}^{(B)})_j$ has the same shape $((u_{n, 1}, m^{(n)}_{1, j}), ..., (u_{n, s_n}, m^{(n)}_{s_n, j}))$.
Therefore, there are permutation homeomorphisms $X^{u_{n+1}, j} \to X^{u_{n+1}, j}$, $j=1, ..., s_{n+1}$, which induce an isomorphism $\sigma_{n+1}: A_{n+1} \to A_{n+1}$ satisfying \eqref{comm-pert-cord}, as desired.
\end{proof}

%\begin{prop}\label{iso-non-simple}
%With notation as above, $$\varinjlim (G^{(1)}_n, \phi^{(1)}_n) \cong \varinjlim (G^{(2)}_n, \phi^{(2)}_n) \quad \Longrightarrow \quad  A(X, (G^{(1)}_n, \phi^{(1)}_n))) \cong A(X, (G^{(2)}_n, \phi^{(2)}_n))).$$
%\end{prop}
%
%\begin{proof}
%Intertwining argument.
%\end{proof}

\subsection{Adding point evaluations and the function $r_\infty^{(0)}$}

Let us assume $X$ is K-contractible (i.e., $\mathrm{K}_*(\mathrm{C}(X)) \cong \mathrm{K}_*(\Comp)$, $*=0, 1$) and 
$$0 < \mathrm{dim}(X) < \infty$$
in the rest of the paper.

Let a finite set $E^{(n)}_{i, j} \subseteq X^{u_{n, i}}$ be given for each $n$, each $1\leq i \leq s_{n}$, and each $1\leq j \leq s_{n+1}$. Set 
$$ \bigoplus_{i=1}^{s_n} \mathrm{M}_{\tilde{u}_{n, i}}(\mathrm{C}(X^{u_{n, i}})) = A_n, $$
where $\tilde{u}_{n, i}$ is defined recursively by
$$\tilde{u}_{n, i} = \sum_{i'=1}^{s_{n-1}} (m^{(n-1)}_{i', i } + \abs{E^{(n-1)}_{i', i}}) \tilde{u}_{n-1, i'},\quad 1\leq i \leq s_n,$$ for $n> 1$,   and $$\tilde{u}_{1, i} = u_{1, i},\quad i=1, ..., s_1.$$

Define a map $\varphi_n: A_n \to A_{n+1}$ by
$$(f_1, ..., f_{s_n}) \mapsto \bigoplus_{j=1}^{s_{n+1}} \mathrm{diag}\{ \underbrace{f_1 \circ \pi^{(1)}_1 ,..., f_1\circ \pi^{(1)}_{m^{(n)}_{1, j}}}_{m^{(n)}_{1, j}}, f_1(E^{(n)}_{1, j}) ,...,  \underbrace{f_{s_n} \circ \pi^{(s_n)}_1 ,..., f_{s_n}\circ \pi^{(s_n)}_{m^{(n)}_{s_n, j}}}_{m^{(n)}_{s_n, j}}, f_{s_n}(E^{(n)}_{s_n, j})\}.$$

Define the $\frac{1}{2}$-dimension ratios at stage $n$ of the sequence $(A_i \to A_{i+1})$ to be
$$r_{n, j}:=\frac{\mathrm{dim}(X)}{2} \cdot \frac{u_{n, j}}{\tilde{u}_{n, j}} = \frac{\mathrm{dim}(X)}{2} \cdot  \frac{\sum_{i=1}^{s_{n-1}} m^{(n-1)}_{i, j } u_{n-1, i} }{\sum_{i=1}^{s_{n-1}} (m^{(n-1)}_{i, j } + \abs{E^{(n-1)}_{i, j}}) \tilde{u}_{n-1, i} }, \quad j=1, ..., s_n.$$ 
Note that $\Kzero(A_n) \cong \Int^{s_n}$. Then, denote by $$r^{(0)}_n:=(r_{n, 1}, ..., r_{n, s_n})$$ the corresponding  continuous affine function on $\mathrm{S}_u(\Kzero(A_n))$, and then regard $r_n^{(0)}$ as an element of $\mathrm{Aff}(\mathrm{S}_u(\Kzero(A)))$. (Note that $\aff(\mathrm{S}_u(\Kzero(A))) = \varinjlim\mathrm{Aff}(\mathrm{S}_u(\Kzero(A_n)))$ canonically.) % (note that $X$ is assume to be connected).

For each $i=1, 2, ...$ and $j > i$, denote the (coordinate) multiplicity matrices of the partial maps $\phi_{i, j}$ and $\varphi_{i, j}$ by $[\phi_{i, j}]$ and $[\varphi_{i, j}]$ respectively, and denote by $[E_{i, j}]$ the multiplicity matrix of the point evaluation maps between $A_i$ and $A_j$. Note that $$[\varphi_{i, j}] = [\phi_{i, j}] + [E_{i, j}], \quad i=1, 2, ...,\ i < j, $$
$$ u_i = [\phi_{1, i}](u_1) \quad \mathrm{and} \quad \tilde{u}_i = [\varphi_{1, i}](u_1),\quad i=2, 3, ... ,$$
and $$\tilde{u}_1 = u_1.$$ 
Then
\begin{eqnarray*}
r_i^{(0)} & = & \frac{\mathrm{dim}(X)}{2} (\frac{u_{i, j}}{\tilde{u}_{i, j}})_{1\leq j\leq s_i} =  \frac{\mathrm{dim}(X)}{2} \frac{[\phi_{1, i}](u_1)}{[\varphi_{1, i}](u_{1})} \\
& = & \frac{\mathrm{dim}(X)}{2} \frac{[\phi_{1, i}](u_1)}{([\phi_{1, i}] + [E_{1, i}])(u_{1})} \\
& = & \frac{\mathrm{dim}(X)}{2} \frac{[\phi_{i-1, i}] ( \cdots  [\phi_{1, 2}](u_1)\cdots)}{([\phi_{i-1, i}] + [E_{i-1, i}]) ( \cdots  ([\phi_{1, 2}] + [E_{1, 2}])(u_{1})\cdots) } ,
\end{eqnarray*}
where the ratio of two vectors means the vector of individual fractions. So, it is clear from the last expression that $r_i$, $i=1, 2, ...$, regarded as a sequence in $\aff(\mathrm{S}_u(\Kzero(A)))$, is decreasing.

Now, let us ensure that the point evaluation sets $E^{(n)}_{i, j}$, $i=1, ..., s_{n-1}$, $j=1, ..., s_n$, are sufficiently small that 
\begin{equation}\label{rapid-growth-cond}
\textrm{$(r_i^{(0)})$ converges uniformly to a strictly positive function $r^{(0)}_\infty \in \aff(\mathrm{S}_u(\Kzero(A)))$.}
\end{equation}
(Recall that $\aff(\mathrm{S}_u(\Kzero(A))) = \varinjlim\mathrm{Aff}(\mathrm{S}_u(\Kzero(A_n))).$)

\begin{rem}
Condition \eqref{rapid-growth-cond} is equivalent to the continuity and strict positivity of the function $r^{(0)}_\infty$.

If Condition \eqref{rapid-growth-cond} is satisfied, then the algebra $A$ is not $\mathcal Z$-absorbing. On the other hand, Condition \eqref{rapid-growth-cond} in general can fail in the way that only some of the dimension ratios $\frac{u_{i, j}}{\tilde{u}_{i, j}}$, $1\leq j \leq s_i$, converge to $0$ as $i\to\infty$ (hence $r_\infty^{(0)}(\tau) = 0$ for some $\tau \in \mathrm{S}_u(\Kzero(A))$) but the function $r_\infty^{(0)}(\tau)$ is not constant equal to zero. In this case, the algebra $A$ is still not $\mathcal Z$-absorbing.

But for technical reasons, let us assume the continuity and strict positivity of $r^{(0)}_\infty$ (Condition \eqref{rapid-growth-cond}) in this paper.
\end{rem}

\begin{rem}
Let $\Delta$ be a metrizable Choquet simplex, and let $\rho$ be a strictly positive continuous affine function on $\Delta$. Is there an AF-Villadsen algebra $A$ such that $(r_\infty^{(0)}, \mathrm{S}_u(\Kzero(A)) \cong (\rho, \Delta)$?
\end{rem}

Since the sequence $(r^{(0)}_i)$ converges uniformly, the function $r^{(0)}_\infty$ is continuous. By compactness of $\mathrm{S}_u(\Kzero(A))$, there is $\delta > 0$ such that $$r^{(0)}_i(\tau) \geq r^{(0)}_\infty(\tau) \geq \delta,\quad \tau \in \mathrm{S}_u(\Kzero(A)),\ i=1, 2, ... .$$ This translates to
\begin{equation}\label{rapid-growth-cond-1-1}
\frac{\mathrm{dim}(X)}{2} \frac{[\phi_{1, i}](u_1)}{([\phi_{1, i}] + [E_{1, i}])(u_{1})} \geq \delta, % \frac{[\phi_i] \circ \cdots \circ [\phi_{1}](u_1)}{([\phi_{i}] + [E_{i}]) \circ \cdots \circ ([\phi_{1}] + [E_{1}])(u_{1})} \geq \delta,
\quad i=1, 2, ..., 
\end{equation}
where $``\geq \delta"$ means each entry of the vector is larger than $\delta$.

Note that 
\begin{eqnarray*} 
r^{(0)}_{i} - r^{(0)}_{i+k} & = &  \frac{\mathrm{dim}(X)}{2} (\varphi_{i, i+k}^*(\frac{[\phi_{1, i}](u_1)}{([\phi_{1, i}] + [E_{1, i}])(u_{1})}) - \frac{[\phi_{1, i+k}](u_1)}{([\phi_{1, i+k}] + [E_{1, i+k}])(u_{1})})  \\
& = & \frac{\mathrm{dim}(X)}{2} (\frac{([\phi_{i, i+k}] + [E_{i, i+k}]) \circ [\phi_{1, i}](u_1)}{([\phi_{i, i+k}] + [E_{i, i+k}])\circ([\phi_{1, i}] + [E_{1, i}])(u_{1})} \\
&& - \frac{[\phi_{i, i+k}] \circ [\phi_{1, i}](u_1)}{([\phi_{i, i+k}] + [E_{i, i+k}])\circ([\phi_{1, i}] + [E_{1, i}])(u_{1})}) \\
& = & \frac{\mathrm{dim}(X)}{2} \frac{[E_{i, i+k}] \circ [\phi_{1, i}](u_1)}{([\phi_{i, i+k}] + [E_{i, i+k}])\circ([\phi_{1, i}] + [E_{1, i}])(u_{1})} \in \aff(\mathrm{S}_u(\Kzero(A_{i+k}))).
\end{eqnarray*}
By the Cauchy criterion, the uniform convergence of $(r^{(0)}_i)$ translates to the condition
\begin{equation}\label{rapid-growth-cond-1-2}
\lim_{i \to\infty}\sup_{k}\norm{ \frac{[E_{i, i+k}] \circ [\phi_{1, i}](u_1)}{([\phi_{i, i+k}] + [E_{i, i+k}])\circ([\phi_{1, i}] + [E_{1, i}])(u_{1})} }_\infty = 0.
\end{equation}

%
%\begin{equation}\label{rapid-growth-cond}
%\norm{\tilde{r}_n - \tilde{r}_{n-1}}_\infty < \frac{1}{2^n}\cdot \frac{\mathrm{dim}(X)}{2},\quad n=1, 2, ...,
%\end{equation}
%and hence the sequence $\tilde{r}_n$, $n=1, 2, ...,$ converges in $(\aff(\mathrm{S}_u(\Kzero(A))), \norm{\cdot}_\infty)$. 
%Let us also assume that $\abs{E_{i, j}^{n}}$ satisfy the following condition:
%
%
%\begin{equation}
%\prod_{n=1}^\infty(\frac{\sum_{i=1}^{s_{n}}m_{i, j}^{(n)} }{\sum_{i=1}^{s_{n}}(m_{i, j}^{(n-1)} + \abs{E_{i, j}^{(n)}})}) > 0
%\end{equation}
%
%For any $\delta>0$, there is $N>0$ such that
%%\begin{equation}
%%(1 - \delta) [\varphi_{i, j}](\mathbf 1_{s_i}) < [\phi_{i, j}](\mathbf 1_{s_i}) <  [\varphi_{i, j}](\mathbf 1_{s_i}),\quad i>N,\  j> i,
%%\end{equation}
%\begin{equation}\label{rapid-growth-cond-pert-1}
%([\varphi_{i, j}] - [\phi_{i, j}])(\mathbf 1_s) < \delta [\varphi_{i, j}](\mathbf 1_{s_i})
%\end{equation}
%where $$\mathbf 1_{s_i} = (\underbrace{1, ..., 1}_{s_i}), $$ and $[\phi_{i, j}]$ and $[\varphi_{i, j}]$ are the multiplicity matrices of $\phi_{i, j}$ and $\varphi_{i, j}$ respectively. (This condition is technically required for the intertwining argument of Proposition \ref{} later, and this condition is automatically satisfied in the UHF case for non-$\mathcal Z$-absorbing Villadsen algebras.---This condition probably is not needed!)

\begin{lem}
With the condition \ref{rapid-growth-cond}, one has 
\begin{equation}\label{rapid-growth-cond-1-3}
\lim_{i \to\infty}\sup_{k}\norm{ \frac{[E_{i, i+k}] \circ ([\phi_{1, i}] + [E_{1, i}])(u_1)}{([\phi_{i, i+k}] + [E_{i, i+k}])\circ([\phi_{1, i}] + [E_{1, i}])(u_{1})} }_\infty = 0,
\end{equation}
which, by definition, may be written as  
\begin{equation}\label{rapid-growth-cond-2}
\lim_{i\to\infty}\sup_{k} \max\{ \frac{([E_{i, i+k}](\tilde{u}_{i}))_j }{([\varphi_{i, i+k}](\tilde{u}_{i}))_j}: j=1, ..., s_{i+k} \} = 0.
\end{equation}
\end{lem}
\begin{proof}
By \eqref{rapid-growth-cond-1-1}, one has
$$ ([\phi_{1, i}] + [E_{1, i}])(u_1) \leq \frac{\mathrm{dim}(X)}{2} \frac{1}{\delta}[\phi_{1, i}](u_1),\quad i=1, 2, ...,$$
and therefore, for each $i, k=1, 2, ...$,
$$ 
\frac{[E_{i, i+k}] \circ ([\phi_{1, i}] + [E_{1, i}])(u_1)}{([\phi_{i, i+k}] + [E_{i, i+k}])\circ([\phi_{1, i}] + [E_{1, i}])(u_{1})} \leq \frac{\mathrm{dim}(X)}{2} \frac{1}{\delta} \frac{[E_{i, i+k}] \circ [\phi_{1, i}](u_1)}{([\phi_{i, i+k}] + [E_{i, i+k}])\circ([\phi_{1, i}] + [E_{1, i}])(u_{1})}.
$$
Thus, by \eqref{rapid-growth-cond-1-2},
\begin{eqnarray*}
 & & \lim_{i \to\infty}\sup_{k}\norm{ \frac{[E_{i, i+k}] \circ ([\phi_{1, i}] + [E_{1, i}])(u_1)}{([\phi_{i, i+k}] + [E_{i, i+k}])\circ([\phi_{1, i}] + [E_{1, i}])(u_{1})} }_\infty \\
 & \leq & \frac{\mathrm{dim}(X)}{2}  \frac{1}{\delta} \lim_{i \to\infty}\sup_{k}\norm{ \frac{[E_{i, i+k}] \circ [\phi_{1, i}](u_1)}{([\phi_{i, i+k}] + [E_{i, i+k}])\circ([\phi_{1, i}] + [E_{1, i}])(u_{1})} }_\infty \\
 &  = & 0.
 \end{eqnarray*}
which is \eqref{rapid-growth-cond-1-3}. 
\end{proof}

%By \eqref{rapid-growth-cond-1-1}, this implies the following condition:
%\begin{equation*}
%\lim_{i \to\infty}\sup_{k}\norm{ \frac{[E_{i, i+k}] \circ ([\phi_{1, i}] + [E_{i, i}])(u_1)}{([\phi_{i, i+k}] + [E_{i, i+k}])\circ([\phi_{1, i}] + [E_{1, i}])(u_{1})} }_\infty = 0,
%\end{equation*}
%which is equivalent to 
%\begin{equation}\label{rapid-growth-cond-2}
%\lim_{i\to\infty}\sup_{k} \max\{ \frac{([E_{i, i+k}]\tilde{u}_{i})_j }{([\varphi_{i, i+k}]\tilde{u}_{i})_j}: j=1, ..., s_{i+k} \} = 0,
%\end{equation}
%\begin{equation}\label{rapid-growth-cond-2}
%\lim_{i\to\infty}\sup_{k > n} \max\{ \frac{\sum_{i=1}^{s_n} \abs{E_{i, j}^{(n, k)}}\tilde{u}_{n, i} }{\sum_{i=1}^{s_n} (m^{(n, k)}_{i, j} + \abs{E^{(n, k)}_{i, j}})\tilde{u}_{n, i}  }: j=1, ..., s_k \} = 0,
%\end{equation}
%where $(m_{i, j}^{n, k})$ is the multiplicity matrix of $\phi_{n, k}$ and $E_{i, j}^{(n, k)}$, $i=1, ..., s_n$, $j=1, ..., s_k$,  are the evaluation points of the map $\varphi_{n, k}$.

Note that $$0 < r^{(0)}_\infty(\tau) < \frac{\mathrm{dim}(X)}{2},\quad \tau \in \mathrm{S}_u(\Kzero(A)).$$ 

%\begin{defn}
%Let $A$ be a unital C*-algebra, and let $p \in A \otimes \mathcal K$ be a projection. The radius of comparison scaled at $p$, denote by $\mathrm{rc}(A, p)$, is defined to be the infimums of
%$$\{r \in \Real : \mathrm{d}_\tau(a) + r \cdot \mathrm{d}_\tau(p) < \mathrm{d}_\tau(b) \ \Rightarrow \ a \precsim b\}.$$
%\end{defn}
%
%\begin{defn}
%Let $A$ be a C*-algebra with $\tr(A) \neq \O$. For each $\tau \in \tr(A)$, define $$r(\tau) = \sup\{r \in(0, \infty): \mathrm{d}_\tau(a) + r < \mathrm{d}_\tau(b)\ \Rightarrow\ a \precsim b \}.$$ Let us call $\tau \mapsto r(\tau)$ the comparison function.
%\end{defn}

\begin{rem}\label{identification-K-base}
The function $r^{(0)}_\infty$ also can be regarded as an affine function on $\mathrm{T}^+(\Kzero(A))$, the cone of all positive homomorphisms $\Kzero(A) \to \Real$ (the simplex $\mathrm{S}_u(G)$ is a base for $\mathrm{T}^+(\Kzero(A))$).
\end{rem}

\begin{example}[\cite{HP-Villadsen}]\label{HP-example}

Let $X$ be K-contractible. Consider two UHF-Villadsen algebras $A^{(1)}:=A(X, (n^{(1)}_i), (k^{(1)}_i))$ and  $A^{(2)}:=A(X, (n^{(2)}_i), (k^{(2)}_i))$. Following \cite{HP-Villadsen}, introduce $k_i^{(1, 2)}$  point evaluations from $A^{(1)}_{i}$ to $A^{(2)}_{i+1}$ %by  $k_i^{(1, 2)}$, 
and $k_i^{(2, 1)}$ point evaluations from $A^{(2)}_{i}$ to $A^{(1)}_{i+1}$. Thus, the multiplicity matrices for the connecting maps $\phi_i$ and $\varphi_i$, before and after adding point evaluations, 
are given by
$$[\phi_i] = \left( \begin{array}{cc} n_i^{(1)} & \\ & n_i^{(2)} \end{array} \right) 
\quad \mathrm{and} \quad
[\varphi_i] = \left( \begin{array}{cc} n_i^{(1)} & \\ & n_i^{(2)} \end{array} \right) + \left( \begin{array}{cc} k_i^{(1)} & k_i^{(2, 1)} \\ k_i^{(1, 2)} & k_i^{(2)} \end{array} \right).
$$
Hence,
$$
\left( 
\begin{array}{c}
u_{i, 1} \\
u_{i, 2}
\end{array}
\right)
=
\left( 
\begin{array}{cc}
n_{i-1}^{(1)} & \\
 & n_{i-1}^{(2)}
\end{array}
\right)
\cdots
\left( 
\begin{array}{cc}
n_1^{(1)} & \\
 & n_1^{(2)}
\end{array}
\right)
\left( 
\begin{array}{c}
1 \\
1
\end{array}
\right)
$$
and
$$
\left( 
\begin{array}{c}
\tilde{u}_{i, 1} \\
\tilde{u}_{i, 2}
\end{array}
\right)
=
\left( 
\begin{array}{cc}
n_{i-1}^{(1)} + k_{i-1}^{(1)} & k_{i-1}^{(2, 1)} \\
k_{i-1}^{(1, 2)} & n_{i-1}^{(2)} + k_{i-1}^{(2)}
\end{array}
\right)
\cdots
\left( 
\begin{array}{cc}
n_1^{(1)} + k_1^{(1)} & k_1^{(2, 1)} \\
k_1^{(1, 2)} & n_1^{(2)} + k_1^{(2)}
\end{array}
\right)
\left( 
\begin{array}{c}
1 \\
1
\end{array}
\right).
$$

Note that the affine map $$[\varphi_i]_0^*: [0, 1] \cong \mathrm{S}_u(\Kzero(A_i))  \leftarrow \mathrm{S}_u(\Kzero(A_{i+1})) \cong [0, 1]$$ is determined by the extreme point assignments 
$$ [0, 1] \ni \frac{k_i^{(2, 1)} \tilde{u}_{i, 2}}{(n^{(1)}_i + k^{(1)}_i)\tilde{u}_{i, 1} + k_i^{(2, 1)} \tilde{u}_{i, 2}}  \mapsfrom 0 \quad \mathrm{and} \quad   [0, 1] \ni \frac{(n^{(2)}_i + k^{(2)}_i)\tilde{u}_{i, 2} }{k_i^{(1, 2)} \tilde{u}_{i, 1} + (n^{(2)}_i + k^{(2)}_i)\tilde{u}_{i, 2}} \mapsfrom 1. $$
Define the compression coefficient
$$c_i = \frac{(n^{(2)}_i + k^{(2)}_i)\tilde{u}_{i, 2} }{k_i^{(1, 2)} \tilde{u}_{i, 1} + (n^{(2)}_i + k^{(2)}_i)\tilde{u}_{i, 2}} - \frac{k_i^{(2, 1)} \tilde{u}_{i, 2}}{(n^{(1)}_i + k^{(1)}_i)\tilde{u}_{i, 1} + k_i^{(2, 1)} \tilde{u}_{i, 2}}. $$

Choose $k_i^{(2, 1)}$, $k_i^{(1, 2)}$, $i=1, 2, ...$, sufficiently small that 
$$c_1 c_2 \cdots >0,$$
which implies that
\begin{equation}\label{farther-prod}
 \lim_{i\to\infty} (c_i c_{i+1} \cdots) = 1.
\end{equation} 
(This ensures that the simplex $\mathrm{S}_u(\Kzero(A))$ does not collapse to a single point, and hence $\mathrm{S}_u(\Kzero(A)) \cong [0, 1]$.)

Write the extreme points of $\mathrm{S}_u(\Kzero(A)) \cong \varprojlim([0, 1], \varphi_i^*)$ as
$$ \tau_1=(s_1, s_2, ...)  \in \prod_{i=1}^\infty [0, 1] \quad \mathrm{and} \quad \tau_2=(t_1, t_2, ...) \in \prod_{i=1}^\infty [0, 1],$$
where 
$s_1 < t_1,\ s_2 < t_2, ... .$ 
Then
$$t_i - s_i = c_i c_{i+1} \cdots, \quad i=1, 2, ... ,$$
and, by \eqref{farther-prod}, 
$$\lim_{i\to\infty} s_i = 0 \quad \mathrm{and} \quad  \lim_{i\to\infty} t_i = 1. $$

Let us calculate the function $r_\infty^{(0)}$.  For each $i = 1, 2, ...$, 
$$r^{(0)}_i(\tau_1) = \frac{1}{2}\mathrm{dim}(X) \cdot (\frac{u_{i, 1}}{\tilde{u}_{i, 1}}(1-s_i) + \frac{u_{i, 2}}{\tilde{u}_{i, 2}} \cdot s_i) \quad \mathrm{and} \quad r^{(0)}_i(\tau_2) = \frac{1}{2}\mathrm{dim}(X) \cdot (\frac{u_{i, 1}}{\tilde{u}_{i, 1}}(1-t_i) + \frac{u_{i, 2}}{\tilde{u}_{i, 2}} \cdot t_i ).$$
Hence,
$$r^{(0)}_\infty(\tau_1) = \lim_{i\to\infty}r^{(0)}_i(\tau_1) = \frac{1}{2}\mathrm{dim}(X)  \cdot \lim_{i\to\infty} \frac{u_{i, 1}}{\tilde{u}_{i, 1}} \quad \mathrm{and} \quad r^{(0)}_\infty(\tau_2) = \lim_{i\to\infty}r^{(0)}_i(\tau_2) = \frac{1}{2}\mathrm{dim}(X) \cdot \lim_{i\to\infty} \frac{u_{i, 2}}{\tilde{u}_{i, 2}}.$$

\end{example}

%\begin{thm}\label{comparison-property-0}
%The function $r_\infty$, regarded as a continuous affine function on $\mathrm{T}^+(A)$, has the following property:
%\begin{equation}\label{gap-function}
% \mathrm{d}_\tau(a) +  r_\infty([\tau]) < \mathrm{d}_\tau(b),\quad \tau \in \mathrm{T}^+(A)\setminus\{0\} \quad \Rightarrow \quad a \precsim b.
%\end{equation}
%
%On the other hand, assume that $X$ is a solid space. If $h \in \mathrm{Aff}(\mathrm{T}^+(\Kzero(A)))$ and if there is $\tau \in \mathrm{T}^+(A)$ such that $$h([\tau]) < r_\infty([\tau]),$$ then there are $a, b \in (A\otimes \mathcal K)^+$ such that $$\mathrm{d}_\tau(a) + h([\tau]) < \mathrm{d}_\tau(b),\quad \tau \in \mathrm{T}^+(A)\setminus\{0\},$$ but $a$ is not Cuntz subequivalent to $b$. Hence $r_\infty$ is the (unique) minimum among the continuous affine functions which satisfy \eqref{gap-function}.
%
%\end{thm}

\section{The comparison property of $r^{(0)}_\infty$}\label{section-comparison-property}

In this section, let us study the comparison properties of the function $r^{(0)}_\infty$. It turns out that the function $r^{(0)}_\infty$ is the smallest continuous affine function on the state space of the order-unit $\Kzero$-group which guarantees comparison. Thus, the function $r^{(0)}_\infty$ can be recovered from the Cuntz semigroup of the limit algebra $A$:

\begin{thm}\label{comparison-property}
Let $A$ be an AF-Villadsen algebra which satisfies \eqref{rapid-growth-cond} (and hence \eqref{rapid-growth-cond-2}). The function $r^{(0)}_\infty$, regarded as a function on $\tr(A)$, has the following property:
\begin{equation}\label{gap-function}
 \mathrm{d}_\tau(a) +  r^{(0)}_\infty(\tau) < \mathrm{d}_\tau(b),\quad \tau \in \tr(A) \quad \Rightarrow \quad a \precsim b,\quad a, b \in A\otimes \mathcal K.
\end{equation}
%where $\tr^+(A)$ denotes the cone of traces of $A$.

On the other hand, assume that $X$ is connected %K-contractible (i.e., $\mathrm{K}_*(\mathrm{C}(X)) = \mathrm{K}_*(\Comp)$, $*=0, 1$) 
and is a finite dimensional solid space (i.e.,  it contains a Euclidean ball of dimension $\mathrm{dim}(X)$). If $h \in \mathrm{Aff}(\mathrm{S}_u(\Kzero(A)))$ and if there is $\tau \in \tr(A)$ such that $$h(\tau) < r^{(0)}_\infty(\tau),$$ then there are $a, b \in (A\otimes\mathcal K)^+$ such that $$\mathrm{d}_\tau(a) + h(\tau) < \mathrm{d}_\tau(b),\quad \tau \in \tr(A),$$ but $a$ is not Cuntz subequivalent to $b$. Thus $r^{(0)}_\infty$ is the (unique) minimum among the continuous affine functions which satisfy \eqref{gap-function} and factor through $\mathrm{T}(A) \to \mathrm{S}_u(\Kzero(A))$.

%On the other hand, if $h \in \mathrm{Aff}(\mathrm{S}_u(\Kzero(A)))$ satisfying $h < r_\infty$ (i.e., $h \neq r_\infty$ and $h([\tau]) \leq r_\infty([\tau])$, $\tau \in \tr(A)$), then there are $a, b \in A^+$ such that
%$$\mathrm{d}_\tau(a) + h([\tau]) < \mathrm{d}_\tau(b)$$
%but $a$ is not Cuntz subequivalent to $b$.
\end{thm}

\begin{proof}

Let $a, b \in (A\otimes \mathcal K)^+$ be positive elements satisfying $$ \mathrm{d}_\tau(a) + r^{(0)}_\infty(\tau) < \mathrm{d}_\tau(b),\quad \tau \in \tr(A).$$ 

Fix an arbitrary $\eps>0$ for the time being.
Since the function $r^{(0)}_\infty$ is continuous and the rank function $\tau \mapsto \mathrm{d}_\tau(b)$ is lower semicontinuous, % and strictly positive, 
by the simplicity of $A$ and the compactness of $\mathrm{T}(A)$,  there is $\delta>0$ such that
$$ \mathrm{d}_\tau((a-\eps)_+) + r^{(0)}_\infty(\tau)+ \delta < \mathrm{d}_\tau(b),\quad \tau \in \tr(A).$$
Since $(r_k)$ converges to $r^{(0)}_\infty$ uniformly, there is $k \in \mathbb N$ large enough that
$$ \mathrm{d}_\tau((a-\eps)_+) + r^{(0)}_k(\tau) + \frac{\delta}{2} < \mathrm{d}_\tau(b),\quad \tau \in \tr(A),$$
and 
\begin{equation}\label{compress-r}
0 < r^{(0)}_k(\tau) - r^{(0)}_n(\tau) < \frac{\delta}{8},\quad \tau \in \tr(A_n),\ n > k, 
\end{equation}
where $r^{(0)}_k$ and $r^{(0)}_n$ are regarded as elements of $\mathrm{Aff}(\mathrm{S}_u(\Kzero(A_n)))$.

Note that there is a (trivial) projection $q \in A_n$, for a sufficiently large $n$, such that
\begin{equation}\label{small-q} 
r^{(0)}_k(\tau) + \frac{\delta}{8} < \tau(q) < r^{(0)}_k(\tau) + \frac{\delta}{2},\quad \tau \in \tr(A_n),
\end{equation}
and then
$$ \mathrm{d}_\tau( (a-\eps)_+) \oplus q) \leq \mathrm{d}_\tau((a-\eps)_+) + r^{(0)}_k(\tau) + \frac{\delta}{2} < \mathrm{d}_\tau(b),\quad \tau \in \tr(A).$$ 

Since $A$ is simple, this implies that there is $N\in \mathbb N$ such that $$(N+1)([(a-\eps)_+] + [q]) < N[b],$$ where $[\cdot]$ denotes the Cuntz class. (See the proof of Proposition 3.2 of \cite{RorUHF-II}.)  %for some projection $r_k$ of \eqref{proj-r}.
Then, by Lemma 5.6 of \cite{Niu-MD}, there are $a_n, b_n, \in A_n$, for $n$ sufficiently large, such that
$$ \norm{a_n - (a-\eps)_+} <\eps, \quad \norm{b_n - b} < \eps, \quad b_n \precsim b, $$
and
$$(N+1)([a_n] + [q]) < N[b_n],$$
which implies 
\begin{equation}\label{gap-q}
\mathrm{tr}(a_n(x)) + \mathrm{tr}(q(x)) < \mathrm{tr}(b_n(x)),\quad x \in X^{u_{n, j}},\quad j=1, ..., s_n,
\end{equation}
where $\mathrm{tr}$ denotes the normalized trace of a matrix algebra.

Note that, by \eqref{compress-r} and \eqref{small-q}, 
$$
r^{(0)}_n(\tau) \approx_{\delta/8} r^{(0)}_k(\tau) < \tau(q) - \frac{\delta}{8},\quad \tau \in \tr(A_n),
$$
and hence
$$r^{(0)}_n(\tau) < \tau(q),\quad \tau \in \tr(A_n).$$
By \eqref{gap-q}, 
$$ \mathrm{tr}(a_n(x)) + r_n(\mathrm{tr}_x) < \mathrm{tr}(a_n(x))  + \mathrm{tr}(q(x)) <  \mathrm{tr}(b_n(x)),\quad x \in X^{u_{n, j}},\ j=1, ..., s_n.$$
%where $\mathrm{tr}$ denotes the normalized trace of a matrix algebra.
Since
$$ r^{(0)}_{n, j} = \frac{\mathrm{dim}(X)}{2} \cdot \frac{{u_{n, j}}}{\tilde{u}_{n, j}},\quad j=1, ..., s_n,$$
this implies 
\begin{equation}\label{rank-comp}
\mathrm{rank}(a_n(x)) + \frac{1}{2}\mathrm{dim}(X^{u_{n, j}}) < \mathrm{rank}(b_n(x)),\quad x \in X^{u_{n, j}},\ j=1, ..., s_n.
\end{equation} 
By Theorem 4.6 of \cite{Toms-Comp-DS}, one has $a_n \precsim b_n$, and hence
$$(a-2\eps)_+ \precsim ((a-\eps)_+-\eps)_+ \precsim a_n \precsim b_n \precsim b.$$
Since $\eps$ is arbitrary, one has $a \precsim b$. This shows \eqref{gap-function}.

Now, let us show that $r^{(0)}_\infty$ is the smallest continuous affine function (on $\mathrm{S}_u(\Kzero(A))$) that satisfies \eqref{gap-function}.

Let $h \in \mathrm{Aff}(\mathrm{S}_u(\Kzero(A)))$ be such that $h([\tau]) < r^{(0)}_\infty([\tau])$ for some $\tau \in \tr(A)$. 
Set 
$$M = \max\{h(\tau): \tau \in \mathrm{S}_u(\Kzero(A))\}$$
and 
\begin{equation}\label{max-gap-1}
\delta = \sup\{r^{(0)}_\infty(\tau) - h(\tau): \tau \in \mathrm{S}_u(\Kzero(A))\} > 0.
\end{equation}
Since $r_\infty$ and $h$ are continuous, and $\mathrm{S}_u(\Kzero(A))$ is compact, there is $\tau_0 \in \mathrm{S}_u(\Kzero(A))$ such that
\begin{equation}\label{max-gap-2}
r^{(0)}_\infty(\tau_0) - h(\tau_0) = \delta.
\end{equation}

Recall that (since $X$ is connected) $\mathrm{Aff}(\mathrm{S}_u(\Kzero(A)))$ has a standard inductive limit decomposition:
\begin{displaymath}
\xymatrix{
(\Real^{s_1}, \norm{\cdot}_\infty) \ar[r] & (\Real^{s_2}, \norm{\cdot}_\infty) \ar[r] & \cdots \ar[r] & \mathrm{Aff}(\mathrm{S}_u(\Kzero(A))),
}
\end{displaymath}
where the connecting map $\Real^{s_n} \to \Real^{s_{n+1}}$ is given by
\begin{equation}\label{connecting-affine-map}
(t_1, ..., t_{s_n}) \mapsto ( \frac{1}{\tilde{u}_{n+1, 1}}\sum_{i=1}^{s_n} (m_{i, 1}^{(n)} + \abs{E_{i, 1}^{(n)}}) (\tilde{u}_{n, i} t_i), ..., \frac{1}{\tilde{u}_{n+1, s_{n+1}}}\sum_{i=1}^{s_n} (m_{i, s_{n+1}}^{(n)} + \abs{E_{i, s_{n+1}}^{(n)}})(\tilde{u}_{n, i} t_i) %\frac{\sum_{i=1}^{s_n}(m_{i, s_{n+1}}^{(n)} + \abs{E_{i, s_{n+1}}^{(n)}})\tilde{u}_{n, i} t_i }{ \sum_{i=1}^{s_n}(m_{i, s_{n+1}}^{(n)} + \abs{E_{i, s_{n+1}}^{(n)}})\tilde{u}_{n, i}}
). 
\end{equation}

Choose $\eps>0$ sufficiently small that
\begin{equation}\label{small-eps}
\frac{2\eps}{\delta + \frac{3}{4}\eps} < \frac{\delta}{64(M+1)}.
\end{equation}
%Pick $\eps \in (0, \frac{3\delta^2}{256(M+1)})$. 
By \eqref{max-gap-1} and \eqref{max-gap-2}, a compactness argument shows that if  $n$ is sufficiently large, there is $h_n \in \Real^{s_n}$ such that
$$ r^{(0)}_{n, j_n} - h_{n, j_n} \approx_{\eps/3} \delta $$
for some $j_n \in \{1, ..., s_n\}$, 
\begin{equation}\label{approx-1}
\norm{\varphi_{n, \infty}^*(r^{(0)}_n) - r^{(0)}_\infty}_\infty < \eps/3 \quad \mathrm{and} \quad \norm{\varphi_{n, \infty}^*(h_n) - h} _\infty < \eps/3,
\end{equation}
and
%$$ \delta - \eps < r_{n, j} - h_{n, j},\quad j=1, ..., s_n. $$
\begin{equation}\label{gap-upper-bd}
  r^{(0)}_{n, j} - h_{n, j} < \delta + \eps,\quad j=1, ..., s_n. 
  \end{equation}

Also assume $n$ is sufficiently large that 
\begin{equation}\label{upper-bd-1}
h_{n, j} + \frac{\delta}{4} < (M + 1)
\end{equation}
and, furthermore, using \eqref{rapid-growth-cond-2}, for any $k > n$,
\begin{equation}\label{far-2}
\frac{1}{\tilde{u}_{k, j}} \sum_{i=1}^{s_n} [E_{n, k}]_{i, j}\tilde{u}_{n, i} < \frac{\delta}{64(M+1)},\quad j=1, ..., s_k.
\end{equation}

By simplicity, one may also assume that $\tilde{u}_{n, j}$, $j=1, ...,  s_n$, are sufficiently large that for any positive real number $t$, there is $d \in \mathbb N$  such that 
\begin{equation}\label{large-u} 
\frac{\delta}{8} < \frac{d}{\tilde{u}_{n, j}} - t < \frac{\delta}{4}. 
\end{equation}

For any $k>n$, consider $h_{k} := \phi^*_{n, k}(h_n)$ and $r_k = \phi_{n, k}^*(r_n)$. Note that there is $j_{k}$ such that
$$r^{(0)}_{k, j_{k}} - h_{k, j_{k}} \approx_\eps \delta. $$
Then, using \eqref{connecting-affine-map}, \eqref{gap-upper-bd}, and the  definition of $S$ below, one has 
\begin{eqnarray*}
\delta &\approx_\eps & r^{(0)}_{k, j_{k}} - h_{k, j_{k}} \\
& = & \frac{1}{\tilde{u}_{k, j_{k}}} \sum_{i=1}^{s_n}(m_{i, j_{k}} + \abs{ E_{i, j_{k}} })(\tilde{u}_{n, i}) (r^{(0)}_{n, i} - h_{n, i}) \\
& = & \frac{1}{\tilde{u}_{k, j_{k}}} \sum_{i\notin S}(m_{i, j_{k}} + \abs{ E_{i, j_{k}} })(\tilde{u}_{n, i}) (r^{(0)}_{n, i} - h_{n, i}) + 
 \frac{1}{\tilde{u}_{k, j_{k}}} \sum_{i \in S}(m_{i, j_{k}} + \abs{ E_{i, j_{k}} })(\tilde{u}_{n, i}) (r^{(0)}_{n, i} - h_{n, i}) \\
& \leq & (\frac{1}{\tilde{u}_{k, j_{k}}} \sum_{i\notin S}(m_{i, j_{k}} + \abs{ E_{i, j_{k}} })(\tilde{u}_{n, i}) )(\delta+ \eps) +  (\frac{1}{\tilde{u}_{k, j_{k}}} \sum_{i\in S}(m_{i, j_{k}} + \abs{ E_{i, j_{k}} })(\tilde{u}_{n, i}) )\frac{\delta}{4}
 \\
& = & (1- \gamma) (\delta+\eps) + \gamma \frac{\delta}{4},
\end{eqnarray*}
where 
$$S:=\{i = 1, ..., s_n: r^{(0)}_{n, i} - h_{n, i} \leq \frac{\delta}{4} \}$$
and
$$\gamma := \frac{1}{\tilde{u}_{k, j_{k}}} \sum_{i\in S}(m_{i, j_{k}} + \abs{ E_{i, j_{k}} })(\tilde{u}_{n, i}). $$
Hence, 
$$(1- \gamma) (\delta + \eps) + \gamma \frac{\delta}{4} > \delta - \eps,$$
which, together with \eqref{small-eps}, implies
$$\gamma < \frac{2\eps}{\delta + \frac{3}{4}\eps} < \frac{\delta}{64(M+1)},$$
or
\begin{equation}\label{small-over}
\frac{1}{\tilde{u}_{k, j_{k}}} \sum_{i\in S}(m_{i, j_{k}} + \abs{ E_{i, j_{k}} })(\tilde{u}_{n, i}) < \frac{\delta}{64(M+1)}.
\end{equation}

By the choice of $\tilde{u}_{n, j}$, $j=1, ..., s_n$ (see \eqref{large-u}), %sufficiently large in a certain sense relative to $\delta$, 
there are natural numbers $d_j$, $j=1, ..., s_n$, such that 
\begin{equation}\label{ratn-appro}
\frac{\delta}{8} < \frac{d_j}{\tilde{u}_{n, j}} - h_{n, j} <  \frac{\delta}{4}. 
\end{equation}
%and, if $h_j < r_j - \frac{\delta}{4}$, then $b_j = \tilde{u}_{n, j}$.
Note that, together with \eqref{upper-bd-1}, one has
\begin{equation}\label{small-d-1}
d_i < (M+1) \tilde{u}_{n, j},\quad j=1, ..., s_n.
\end{equation}

If $j \notin S$ (so that $h_{n, j} < r_{n, j} - {\delta}/{4}$), one has $$\frac{d_j}{\tilde{u}_{n, j}} < h_{n, j} + \frac{\delta}{4} < r^{(0)}_{n, j} = \frac{1}{2}\mathrm{dim}(X)\frac{u_{n, j}}{\tilde{u}_{n, j}},$$ and so $$2d_j < u_{n, j}\mathrm{dim}(X).$$
Since $X$ is solid, there is a $(2d_j+1)$-dimensional Euclidean ball $B_j \subseteq X^{u_{n, j}}$, and there is a complex vector bundle $E_j$ over $\partial B_j$ $(\cong S^{2d_j})$ such that
$$ \mathrm{rank}(E_j) = d_j \quad \mathrm{and} \quad c_{d_j}(E_j) \in \mathrm{H}^{2d_j}(S^{2d_j}) \setminus\{0\}. $$
%(Recall that the total Chern class of E is 1+e.) 
(Such a vector bundle exists, as, otherwise, the $d_j$-th Chern class of every vector bundle would be trivial, and 
then the Chern character would not induce a rational isomorphism between the K-group and the
cohomology group of the sphere $S^{2d_j}$.)

Denote by $p_j \in \mathrm{C}(S^{2d_j}) \otimes \mathcal K$ the projection corresponding to $E_j$, and extended to a positive element of $\mathrm{C}(X^{u_{n, j}}) \otimes \mathcal K$ with rank at least $d_j$.

If $j \in S$, just choose $p_j$ to be a constant function with rank $d_j$. Set $$p = \bigoplus_{j=1}^{s_n} p_j. $$
Then, by \eqref{ratn-appro},
\begin{equation}\label{p-rank-lbd}
\mathrm{d}_{\mathrm{tr}_x}(p) = \frac{\mathrm{rank}(p_j)}{\tilde{u}_{n, j}} \geq \frac{d_j}{ \tilde{u}_{n, j} } > h_{n, j} + \frac{\delta}{8} ,\quad x \in X_{n, j},\ j = 1, ..., s_n.
\end{equation}

Choose $q \in A_n$ to be a trivial projection such that
$$ \frac{\delta}{32} < \tau(q) < \frac{\delta}{16},\quad \tau \in\tr(A_n). $$ Then it follows from \eqref{approx-1} and \eqref{p-rank-lbd} that 
$$\mathrm{d}_\tau(q) + h(\tau) < h(\tau) + \frac{\delta}{16} < h_n(\tau) + \frac{\delta}{8} < \mathrm{d}_\tau(p),\quad \tau \in \tr(A). $$

Let us show that $q$ actually is not Cuntz subequivalent to $p$. 
For any $k > n$, consider the direct summand $A_{k, j_{k}}$, and consider the closed subset 
$$D:=\underbrace{S^{2d_1} \times \cdots \times S^{2d_1}}_{m_{1, j_k}} \times \cdots \times \underbrace{S^{2d_1} \times \cdots \times S^{2d_1}}_{m_{1, j_k}} \subseteq X^{u_{k, j_k}}.$$ 
Then the restriction of $p$ to this closed subset is equivalent to the projection onto a vector bundle with total Chern class non-zero at degree
$$ 2\sum_{i \notin S} {m_{i, j_k}}d_i,$$
and by Remark 3.2 of \cite{ELN-Vill}, this implies that any trivial sub-bundle must have rank at most 
$$ \sum_{i = 1}^{s_n} ({m_{i, j_{k}}} + \abs{E_{i, j_{k}}})d_i - \sum_{i \notin S} {m_{i, j_{k}}}d_i.$$
Then, by \eqref{small-d-1}, \eqref{small-over}, and \eqref{far-2}, the (normalized)  trace of the projection associated to this trivial sub-bundle is at most 
\begin{eqnarray*}
&& \frac{1}{\tilde{u}_{k, j_{k}}}(\sum_{i = 1}^{s_n} ({m_{i, j_{k}}} + \abs{E_{i, j_{k}}})d_i -   \sum_{i \notin S} {m_{i, j_{k}}}d_i) \\
 & = &  \frac{1}{\tilde{u}_{k, j_{k}}}(\sum_{i \in S} ({m_{i, j_{k}}} + \abs{E_{i, j_{k}}})d_i + \sum_{i \notin S} ({m_{i, j_{k}}} + \abs{E_{i, j_{k}}})d_i -   \sum_{i \not \in S} {m_{i, j_{k}}}d_i ) \\
 & = &  \frac{1}{\tilde{u}_{k, j_{k}}}(\sum_{i \in S} ({m_{i, j_{k}}} + \abs{E_{i, j_{k}}})d_i + \sum_{i \notin S} \abs{E_{i, j_{k}}}d_i) \\
 & < & \frac{M+1}{\tilde{u}_{k, j_{k}}}(\sum_{i \in S} ({m_{i, j_{k}}} + \abs{E_{i, j_{k}}})\tilde{u}_{n, i} + \sum_{i =1}^{s_n} \abs{E_{i, j_{k}}}\tilde{u}_{n, i}) \\
& < &  (M+1)(\frac{\delta}{64(M+1)} + \frac{\delta}{64(M+1)}) = \frac{\delta}{32}.
 \end{eqnarray*}
Since the trace of the restriction of $q$ to $D$ is larger than $\delta/32$, this implies that $q$ is not Cuntz subequivalent to $p$, as asserted.
\end{proof}

Since the trace simplex $\tr(A)$ is a base for the cone $\mathrm{T}^+(A)$, Theorem \ref{comparison-property} can be reformulated in terms of $\mathrm{T}^+(A)$ and its affine functions:

\begin{thm}\label{comparison-property-0}
Let $A$ be an AF-Villadsen algebra which satisfies \eqref{rapid-growth-cond}. The function $r^{(0)}_\infty$, regarded as a continuous affine function on $\mathrm{T}^+(A)$, has the following property:
\begin{equation}\label{gap-function-0}
 \mathrm{d}_\tau(a) +  r^{(0)}_\infty(\tau) < \mathrm{d}_\tau(b),\quad \tau \in \mathrm{T}^+(A)\setminus\{0\} \quad \Rightarrow \quad a \precsim b.
\end{equation}

On the other hand, assume that the seed space $X$ is a finite dimensional, connected, and  solid space. If $h \in \mathrm{Aff}(\mathrm{T}^+(\Kzero(A)))$, where $\mathrm{T}^+(\Kzero(A))$ denotes the cone $\Real^+\mathrm{S}_u(\Kzero(A))$,  and if there is $\tau \in \mathrm{T}^+(A)$ such that $$h([\tau]_0) < r^{(0)}_\infty([\tau]_0),$$ then there are $a, b \in (A\otimes \mathcal K)^+$ such that $$\mathrm{d}_\tau(a) + h([\tau]_0) < \mathrm{d}_\tau(b),\quad \tau \in \mathrm{T}^+(A)\setminus\{0\},$$ but $a$ is not Cuntz subequivalent to $b$. Hence $r^{(0)}_\infty$ is the (unique) minimum among the continuous affine functions which factor through $\mathrm{T}^+(A) \to \mathrm{T}^+(\Kzero(A))$ and satisfy \eqref{gap-function-0}.

\end{thm}

\begin{rem}
In Theorem \ref{low-env-gn} and Remark \ref{uniq-crf} of Section \ref{general-gap-functions}, it will be shown that, for a UHF-Villadsen algebra $A$,  the infimum of all continuous affine functions on $\mathrm{T}^+(A)$ which satisfy \eqref{gap-function-0}, but not necessarily factoring through $\mathrm{T}^+(\Kzero(A))$ (gap functions; see Definition \ref{defn-gap-fctn}), still has a similar comparison property. However, the function $r_\infty$ might not be continuous, a possible example being with the seed space $[0, 1] \vee [0, 1]^2$, and does not necessarily factor through $\mathrm{T}^+(\Kzero(A))$, as follows from Corollary \ref{different-seed}.
\end{rem}

\begin{cor}\label{aut-fix}
Let $A$ be an AF-Villadsen algebra with the seed space to be finite dimensional, connected, and solid, and assume  that $A$ satisfies \eqref{rapid-growth-cond}. Let $\sigma \in \mathrm{Aut}(A)$. Then $$r^{(0)}_\infty((([\sigma]_0)^*(\tau))) = r^{(0)}_\infty(\tau),\quad \tau \in \mathrm{S}_u(\Kzero(A)).$$
%Let $A(X, G, \mathcal E)$ and $B(Y, H, \mathcal F)$ be two Villadsen algebras. Then, if $\mathrm{Cu}(A) \cong \mathrm{Cu}(B)$, then $\mathrm{S}_u(\Kzero(A)) \cong \mathrm{S}_u(\Kzero(B))$, and, under this isomorphism,  $r_\infty^{(A)} = r_\infty^{(B)}$ (recall that the construction of the function $r_\infty$ depends on the Brattelli diagram and the point-evaluations).
\end{cor}
\begin{proof}
By Theorem \ref{comparison-property}, the function $r^{(0)}_\infty$ is the smallest function which has the comparison property \eqref{gap-function}, and therefore the function $r^{(0)}_\infty \circ ([\sigma]_0)^*$ is also the smallest function which has the comparison property \eqref{gap-function}. Therefore $r^{(0)}_\infty \circ ([\sigma]_0)^* = r^{(0)}_\infty.$
\end{proof}

%It follows from Theorem \ref{comparison-property} that the radius of comparison of $A$ equals to the maximal value of the function $r_\infty$. Indeed, one can recover the radius of comparison of each unital hereditary sub-C*-algebra in the following way.

It is straightforward, as we shall now show, that
$$\mathrm{rc}(A) = \max\{r^{(0)}_\infty(\tau): \tau \in \tr(A)\}.$$
In fact, the radius of comparison of any unital hereditary sub-C*-algebra of $A\otimes \mathcal K$ can be recovered in a similar way.

\begin{thm}\label{rc-general}
Let $A$ be a simple unital C*-algebra such that $\tr(A) \neq \O$, and let $r \in \mathrm{Aff}(\mathrm{T}^+(A))$ have the following three properties:
\begin{enumerate}
\item The function $r$ factors though $\mathrm{T}^+(A) \to \mathrm{T}^+(\Kzero(A))$,

\item\label{comp-property} The function $r$ has the property that for any $a, b \in (A \otimes \mathcal K)^+$, $$ \mathrm{d}_\tau(a) + r(\tau) < \mathrm{d}_\tau(b),\quad \tau \in\mathrm{T}^+(A)\quad \Longrightarrow \quad a \precsim b.$$

\item $r$ is the smallest element of $\mathrm{Aff}(\mathrm{T}^+(A))$  factoring though $\mathrm{T}^+(A) \to \mathrm{T}^+(\Kzero(A))$ and having  the comparison property (\ref{comp-property}), in the following sense: if $h \in \mathrm{Aff}(\mathrm{T}^+(A))$, $h$  factors through $\mathrm{T}^+(A) \to \mathrm{T}^+(\Kzero(A))$ and $h$ has the comparison property (\ref{comp-property}), then $$r(\tau) \leq h(\tau),\quad \tau \in \mathrm{T}^+(A).$$
%  such that $h(\tau_0) < r(\tau_0)$ for some $\tau_0 \in \mathrm{T}^+(A)$, then $h$ does not have the comparison property (\ref{comp-property}).

\end{enumerate}
Then, for any projection $p \in A \otimes \mathcal K$, one has
$$\mathrm{rc}(p(A \otimes \mathcal K)p) = \max\{r(\tau): \tau(p) = 1,\ \tau \in \mathrm{T}^+(A)\}.$$
\end{thm}

\begin{proof}
Since $p$ is full, the set $$C_p = \{\tau \in \mathrm{T}^+(A): \tau(p) = 1 \} \cong \tr(p(A\otimes \mathcal K)p)$$ is a base for $\mathrm{T}^+(A)$,
and the set 
$$C'_p = \{\tau \in \mathrm{T}^+(\Kzero(A)): \tau([p]) = 1 \} \cong \mathrm{S}_{[p]}(\Kzero(p(A\otimes \mathcal K)p))$$
is a base for $\mathrm{T}^+(\Kzero(A))$. %Since $A$ is exact, the canonical map $C_p \to C'_p$ is surjective.

Let $$s < \max\{r(\tau): \tau(p) = 1,\ \tau \in \mathrm{T}^+(A)\} = \max\{r(\tau): \tau([p]_0) = 1,\ \tau \in \mathrm{T}^+(\Kzero(A))\}  .$$ Regard $s$ as a constant affine function on $C'_p$. Since $C'_p$ is a base for $\mathrm{T}^+(\Kzero(A))$, the affine function $s$ can be extended to a continuous affine function on $\mathrm{T}^+(\Kzero(A))$, $0$ at $0$, and hence a continuous affine function on $\mathrm{T}^+(A)$ (factoring through $\mathrm{T}^+(A) \to \mathrm{T}^+(\Kzero(A))$). Still denote it by $s$. 

Then there is $\tau_0 \in \mathrm{T}^+(A)$ such that $$s(\tau_0) < r(\tau_0).$$ 
By Condition (3), the affine function $s$ does not have the comparison property (\ref{comp-property}), and hence there are positive elements $a, b \in A\otimes \mathcal K$ such that
$$\mathrm{d}_\tau(a) + s(\tau) < \mathrm{d}_\tau(b),\quad \tau \in \mathrm{T}^+(A),$$
but $a$ is not Cuntz subequivalent to $b$. Restricting to $C_p$, one has
$$\mathrm{d}_\tau(a) + s < \mathrm{d}_\tau(b),\quad \tau \in C_p.$$
Since $s$ is arbitrary, 
$$ \max\{r(\tau): \tau(p) = 1,\ \tau \in \mathrm{T}^+(A)\} \leq \mathrm{rc}(p(A\otimes \mathcal K)p).$$

Let us show the reverse inequality. Assume $a, b \in A \otimes \mathcal K$ are positive elements such that
$$\mathrm{d}_\tau(a) +  \max\{r(\tau): \tau(p) = 1,\ \tau \in \mathrm{T}^+(A)\} < \mathrm{d}_\tau(b),\quad \tau \in C_p.$$ 
Then
$$\mathrm{d}_\tau(a) +  r(\tau)  \leq \mathrm{d}_\tau(a) +  \max\{r(\tau): \tau(p) = 1,\ \tau \in \mathrm{T}^+(A)\} < \mathrm{d}_\tau(b),\quad \tau \in C_p.$$ Since $C_p$ is a base for $\mathrm{T}^+(A)$, one has
$$\mathrm{d}_\tau(a) +  r(\tau) < \mathrm{d}_\tau(b),\quad \tau \in \mathrm{T}^+(A).$$ By the comparison property of $r$, one has $a \precsim b$. Therefore,
$$ \max\{r(\tau): \tau(p) = 1,\ \tau \in \mathrm{T}^+(A)\} \geq \mathrm{rc}(p(A\otimes \mathcal K)p),$$
as desired.
\end{proof}

\begin{cor}\label{rc-V}
Let $A$ be an AF-Villadsen algebra which satisfies \eqref{rapid-growth-cond}, with the seed space finite dimensional, solid, and connected. Then, for any projection $p \in A\otimes \mathcal K$, one has
$$\mathrm{rc}(p(A\otimes \mathcal K)p) = \max\{r^{(0)}_\infty(\tau): \tau(p) = 1,\ \tau \in \mathrm{T}^+(A)\}.$$
\end{cor}

\begin{proof}
It follows from Theorem \ref{comparison-property} that the function $r_\infty^{(0)}$ satisfies the assumptions of Theorem \ref{rc-general}, and then the corollary follows from Theorem \ref{rc-general}.
\end{proof}

\begin{example}[\cite{HP-Villadsen}]\label{exm-HP}

Consider the Villadsen algebra of Example \ref{HP-example}. Assume that $$ \lim_{i\to\infty} \frac{u_{i, 1}}{\tilde{u}_{i, 1}} \neq \lim_{i\to\infty} \frac{u_{i, 2}}{\tilde{u}_{i, 2}},$$
and hence that 
$$r^{(0)}_\infty(\tau_1) \neq r^{(0)}_\infty(\tau_2),$$
where $\tau_1$ and $\tau_2$ are the extreme points of $\mathrm{S}_u(\Kzero(A)) \cong [0, 1]$ (see Example \ref{HP-example}).
Then, by Corollary \ref{aut-fix}, there is no automorphism $\sigma: A \to A$ such that $\tau_1 \circ ([\sigma]_0)^* = \tau_2$. That is, there is no automorphism which flips $\mathrm{S}_u(\Kzero(A)) \cong  [0, 1]$.

Let $p \in A\otimes \mathcal K$ be a projection, and let us calculate $\mathrm{rc}(p(A \otimes\mathcal K)p)$. Without loss of generality, one may assume that $p \in A_i$ for some $i \in \mathbb N$ (as, if $p$ is unitarily equivalent to $q$, then the hereditary sub-C*-algebras generated by $p$ and $q$ are isomorphic).  

Note that $$\tr^+(A) = \{\alpha \tau_1 + \beta \tau_2: \alpha, \beta \in [0, +\infty)\},$$ and then consider the section
$$\{\tau \in \tr^+(A): \tau(p) = 1\}.$$
Write $\tau = \alpha \tau_1 + \beta \tau_2$; and then 
$$ 1 = (\alpha \tau_1 + \beta \tau_2)(p) = \alpha\tau_1(p) + \beta \tau_2(p).$$ Since $\beta \geq 0$, a simple calculation shows that
$$ 0 \leq \alpha \leq \frac{1}{\tau_1(p)}. $$
Then, by Corollary \ref{rc-V}, 
 \begin{eqnarray*} 
\mathrm{rc}(p(A \otimes \mathcal K)p)& = & \max\{ r^{(0)}_\infty( \alpha \tau_1 + \beta \tau_2 ): (\alpha \tau_1 + \beta \tau_2)(p) = 1 \} \\
& = & \max \{\alpha r^{(0)}_\infty(\tau_1) + \frac{1}{\tau_2(p)}(1 - \alpha \tau_1(p)) r^{(0)}_\infty(\tau_2) : 0 \leq \alpha \leq \frac{1}{\tau_1(p)} \} \\
& = &  \max \{\frac{r^{(0)}_\infty(\tau_2)}{\tau_2(p)} + (r^{(0)}_\infty(\tau_1) - \frac{\tau_1(p)}{\tau_2(p)} r^{(0)}_\infty(\tau_2))\alpha: 0 \leq \alpha \leq \frac{1}{\tau_1(p)}  \} \\
%& = & \left\{ \begin{array}{ll} 
%\frac{r_\infty(\tau_2)}{\tau_2(p)}, & \textrm{if $r_\infty(\tau_1) \tau_2(p) \leq  {\tau_1(p)} r_\infty(\tau_2)$} \\
%\frac{r_\infty(\tau_1)}{\tau_1(p)}, & \textrm{otherwise}.
% \end{array} \right.\\
 & = & \max\{\frac{r^{(0)}_\infty(\tau_1)}{\tau_1(p)}, \frac{r^{(0)}_\infty(\tau_2)}{\tau_2(p)} \}.
 \end{eqnarray*}

\end{example}

\section{The isomorphism theorem}

In this section, let us show that the AF-Villadsen algebras with a given (finite dimensional, K-contractible, and solid) seed space are classified by their $\Kzero$-group together with the function $r^{(0)}_\infty$.

\begin{prop}[cf.~Lemma 7.4 of \cite{ELN-Vill}]\label{pre-intertwining}
Let $X$ be a K-contractible (i.e., $\mathrm{K}_*(\mathrm{C}(X)) = \mathrm{K}_*(\Comp)$, $*=0, 1$) metrizable compact space such that $0 < \mathrm{dim}(X) < \infty$, and let $A(X, G^{(A)}, \mathcal E^{(A)})$ and $B(X, G^{(B)}, \mathcal E^{(B)})$ be AF-Villadsen algebras with seed space $X$ satisfying \eqref{rapid-growth-cond} (and therefore \eqref{rapid-growth-cond-2}) such  that $$ (\Kzero(A), [1_A]_0,  r_\infty^{(0)}(A)) \cong (\Kzero(B), [1_B]_0, r_\infty^{(0)}{(B)}). $$ Let $\delta_1 > \delta_2 > \cdots$ be a sequence of strictly positive numbers such that $\sum_{i=1}^\infty \delta_i < 1$. Then, on  telescoping, there is a diagram
\begin{equation}
\xymatrix{
A_1 \ar[r]^{\phi_1} \ar[d]_{\rho_1} & A_2 \ar[r]^{\phi_2} \ar[d]_{\rho_2} & \cdots & \\
B_1 \ar[r]^{\psi_1} \ar[ur]^{\eta_1} & B_2 \ar[r]^{\psi_1} & \cdots & 
}
\end{equation}
such that each of $\phi_i, \psi_i, \rho_i, \kappa_i$, $i=1, 2, ...$, consists of independent coordinate projections and point evaluations, i.e., restricted to each direct summand of the domain, it has the form  
$$ (f_1, f_2, ..., f_s) \mapsto \mathrm{diag}\{ f_1 \circ P_1, ..., f_s \circ P_s, \textrm{point evaluations}\},$$
where $P_1, ..., P_s$ are mutually disjoint sets of coordinate projections,
and
for each $i=1, 2, ...$, there are decompositions
$$\phi_i = \mathrm{diag}\{P_{A, i}, R'_{A, i}, \Theta_{A, i}\}, \quad \psi_i = \mathrm{diag}\{P_{B, i}, R'_{B, i}, \Theta_{B, i}\},  $$
and
$$ \eta_i\circ\rho_i =  \mathrm{diag}\{P_{A, i}, R''_{A, i}, \Theta_{A, i}\},\quad  \rho_{i+1}\circ\eta_i =  \mathrm{diag}\{P_{B, i}, R''_{B, i}, \Theta_{B, i}\},$$
where $P_{A, i}: A_i \to A_{i+1}$ and $P_{B, i}: B_i \to B_{i+1}$ consist of coordinate projections, and $\Theta_{A, i}: A_i \to A_{i+1}$ and $\Theta_{B, i}: B_i \to B_{i+1}$ consist of point evaluations, 
such that, for each $i=1, 2, ...$, 
%$$ 
%\abs{
%\tau(\phi_i(h)) - 
%\eta_i(\rho_i(h))
%} < \delta_i, \quad h \in A_i,\ \norm{h} \leq 1,\ \tau \in \tr(A_{i+1}),
%$$
%
%$$ 
%\abs{
%\tau(\psi_i(h)) - 
%\rho_{i+1}(\eta_i(h))
%} < \delta_i, \quad h \in B_i,\ \norm{h} \leq 1,\ \tau \in \tr(B_{i+1}),
%$$
%
%$$
%\mathrm{rank}(\Theta'_{A, i}) = \mathrm{rank}(\Theta''_{A, i}) \quad \mathrm{and} \quad \quad  \frac{\mathrm{rank}(R'_{A, i})}{\mathrm{rank}(\Theta_{A, i})} = \frac{\mathrm{rank}(R''_{A, i})}{\mathrm{rank}(\Theta_{A, i})} < \delta_i
%$$
%and
%$$
%\mathrm{rank}(\Theta'_{B, i}) = \mathrm{rank}(\Theta''_{B, i}) \quad \mathrm{and} \quad \quad  \frac{\mathrm{rank}(R'_{B, i})}{\mathrm{rank}(\Theta_{B, i})} = \frac{\mathrm{rank}(R''_{B, i})}{\mathrm{rank}(\Theta_{B, i})} < \delta_i
%$$
$$ \frac{\mathrm{rank}_j(R'_{A, i}(1_{A_i}))}{\mathrm{rank}_j(\Theta_{A, i}(1_{A_i}))} = \frac{\mathrm{rank}_j(R''_{A, i}(1_{A_i}))}{\mathrm{rank}_j(\Theta_{A, i}(1_{A_i}))} < \delta_i,\quad j=1, ..., s_{i+1}^{(A)},
$$
and
$$
\frac{\mathrm{rank}_j(R'_{B, i}(1_{B_i}))}{\mathrm{rank}_j(\Theta_{B, i}(1_{B_i}))} = \frac{\mathrm{rank}_j(R''_{B, i}(1_{B_i}))}{\mathrm{rank}_j(\Theta_{B, i}(1_{B_i}))} < \delta_i, \quad j=1, ..., s_{i+1}^{(B)}.
 $$

In particular, $\tr(A) \cong \tr(B)$, and in a way compatible with the isomorphism of $\Kzero$-groups.

\end{prop}

\begin{proof}
Since $X$ is K-contractible, one has  
$$ (\Kzero(G^{(A)}), \Kzero^+(G^{(A)}), [\tilde{u}^{(A)}]) \cong  (\Kzero(A), \Kzero^+(A), [1_{A}]_0)$$
and
$$ (\Kzero(G^{(B)}), \Kzero^+(G^{(B)}), [\tilde{u}^{(B)}]) \cong  (\Kzero(B), \Kzero^+(B), [1_{B}]_0).$$
Since $(\Kzero(A), [1_A]_0) \cong (\Kzero(B), [1_B]_0)$, there is an isomorphism $$ \kappa_\infty:  (\Kzero(G^{(A)}), \Kzero^+(G^{(A)}), [\tilde{u}^{(A)}]) \cong (\Kzero(G^{(B)}), \Kzero^+(G^{(B)}), [\tilde{u}^{(B)}]).$$ Therefore, upon a telescoping, there is a commutative diagram
\begin{equation}\label{comm-diag-K0-1}
\xymatrix{
G_1^{(A)} \ar[r] \ar[d]_{\kappa^{[A, B]}_1} & G_2^{(A)} \ar[r] \ar[d] & \cdots \ar[r] & \Kzero(G^{(A)}) \ar[d] \\
G_1^{(B)} \ar[r] \ar[ur]_-{\kappa^{[B, A]}_1}& G_2^{(B)} \ar[r] \ar[ur] & \cdots \ar[r] & \Kzero(G^{(B)}) \ar[u]
}
\end{equation}
which induces the isomorphism $\kappa_\infty$.

Consider the inductive sequences
\begin{displaymath}
\xymatrix{
A_1 \ar[r] & A_2 \ar[r] & \cdots \ar[r] & A, \\
B_1 \ar[r] & B_2 \ar[r] & \cdots \ar[r] & B.
}
\end{displaymath}

Write $r_i^{(0)}{(A)} \in \Real^{s_i^{(A)}}$ and $r_i^{(0)}{(B)} \in \Real^{s^{(B)}_i}$, $i=1, 2, ...$, for the affine functions at stage $i$ which converge (uniformly---see \eqref{rapid-growth-cond}) to $r_\infty^{(0)}{(A)}$ and $r_\infty^{(0)}{(B)}$ respectively.

Consider $\delta_1$ and consider $A_1$. It follows from the construction of $A$ that there is $\Delta_1>0$ such that for each $i \geq 1$, if one decomposes $\phi_{1, i}$ as $$\phi_{1, i} = P \oplus \Theta: A_1 \to A_i,$$ where $P: A_1 \to A_i$ consists of coordinate projections and $\Theta: A_1 \to A_i$ consists of point evaluations, then 
\begin{equation}\label{lower-bd} 
\tau(\Theta(1_{A_1})) > \Delta_1,\quad \tau \in\tr(A_i).
\end{equation}

Choose $\delta_1'$ sufficiently small that $$\frac{2\delta_1'  + (1-(1- 2 \delta_1')^3)}{\Delta_1} < \delta_1. $$ 

Recall from \eqref{rapid-growth-cond} that $(r_i^{(0)}{(A)})$ and $(r_i^{(0)}{(B)})$ converge decreasingly and uniformly to $r_\infty^{(0)}{(A)}$ and $r_\infty^{(0)}{(B)}$ respectively. %there is $i'_1$ such that 
%$$ \abs{r_i^{(A)} - r_\infty^{(A)}} < \delta'_1 \quad \mathrm{and} \quad \abs{r_i^{(B)} - r_\infty^{(B)}} < \delta'_1, \quad i \geq i'_1.$$
Since $r_\infty^{(0)}{(A)} = r_{\infty}^{(0)}{(B)}$ under the isomorphism induced by $\kappa_\infty$, and since $r_\infty^{(0)}{(A)}$ and $r_\infty^{(0)}{(B)}$ are strictly positive, there is $i'_1$ %can be chosen to be sufficiently large 
such that the differences
$$\norm{r_i^{(0)}{(A)} - r_\infty^{(0)}{(A)}}_\infty \quad \mathrm{and} \quad \norm{r_i^{(0)}{(B)} - r_\infty^{(0)}{(B)}}_\infty,\quad i \geq i_1', $$
are small enough that
%$$ 2\delta_1 > (\kappa^{A, B}_{i, j})^*(r_i^{(A)}) - r_j^{(B)} > 0$$
\begin{equation}\label{close-ratio}
(1-\delta'_1)(\kappa^{A, B}_{i, j})^*(r_i^{(0)}{(A)}) < r_j^{(0)}{(B)},\quad j>i \geq i_1'. 
\end{equation}

By \eqref{rapid-growth-cond-2}, $i_1'$ is large enough that
\begin{equation}\label{gr-con-A}
\sum_{i=1}^{s^{(A)}_{i_1'}}([\phi_{i_1', k}]_{i, j} -  (D[\phi_{i'_1, k}])_{i, j})\tilde{u}^{(A)}_{i_1', i} < \delta_1' \sum_{i=1}^{s^{(A)}_{i_1'}} [\phi_{i_1', k}]_{i, j}\tilde{u}^{(A)}_{i_1', i} , \quad j=1, ..., s_{k}^{(A)},\ k> i_1',
\end{equation}
where $D(\cdot)$ denotes the multiplicity matrix of the coordinate projection component.

%\vskip 1in
%
%Define the map $\rho_1: A_1 \to B_1$ by
%\begin{eqnarray*}
%\rho_1(f_1, ..., f_{s^{(A)}_1}) & \mapsto &  ( \mathrm{diag}\{f_1\circ P_{1, 1}, ..., f_{s_1} \circ P_{s_1, 1}, \textrm{point evaluations}\}, ..., \\
%& & \mathrm{diag}\{f_1\circ P_{1, s_2}, ..., f_{s_1} \circ P_{s_1, s_2}, \textrm{point evaluations}\}),
%\end{eqnarray*}
%where each
%$$P^{A, B}_{i, j}: X^{u^{(B)}_{2, j}} \to X^{u^{(A)}_{1, i}},\quad i=1, ..., s^{(A)}_1,\ j=1, ..., s^{(B)}$$
%consists of $$ \delta^{A, B}_{i, j} = ?$$ disjoint coordinate projections, and the $\Kzero$-multiplicity of $\rho_1$ is 
%$$
%[\rho_1]_0 = \kappa_1^{[A, B]}.
%$$
%Denote by $D^{A, B}_1$ the coordinate projection matrix $(\delta^{A, B}_{i, j})_{i, j}$ and denote by $M_1^{A, B}$ the multiplicity matrix of $\kappa_1^{[A, B]}$.
%
%\vskip 1in

Fix $i_1'$. Since the ordered groups $\Kzero(G^{(A)}) \cong \Kzero(G^{(B)})$ are simple, if $i_1'' > i_1'$ is sufficiently large and if $$ (m_{i, j}), \quad i=1, ..., s_{i_1'}^{(A)},\ j=1, ..., s_{i_1''}^{(B)}, $$ denotes the multiplicity matrix of $\kappa_{i_1', i_1''}^{A, B}$, 
then there exist positive integers
$$\delta_{i, j}^{A, B}, \quad i=1, ..., s_{i_1'}^{(A)},\ j=1, ..., s_{i_1''}^{(B)}, $$
satisfying
\begin{equation}\label{app-multi-A}
1-2\delta'_1<\frac{\delta_{i, j}^{A, B}}{m_{i, j}^{A, B}} < 1 - \delta'_1 ,\quad i=1, ..., s_{i_1'}^{(A)},\ j=1, ..., s_{i_1''}^{(B)}.
\end{equation}
Recall that 
$$r_i^{(0)}{(A)} = \frac{1}{2}\mathrm{dim}(X)(\frac{u^{(A)}_{i, j}}{\tilde{u}^{(A)}_{i, j}})_j, \quad  r_i^{(0)}{(B)} = \frac{1}{2}\mathrm{dim}(X)(\frac{u^{(B)}_{i, j}}{\tilde{u}^{(B)}_{i, j}})_j, $$
and $0< \mathrm{dim}(X) < \infty$. Then,  by \eqref{close-ratio},
one has, for each $j = 1, ..., s_{i_1''}^{(B)}$, 
\begin{eqnarray*}
\frac{1}{\tilde{u}_{i_1'', j}^{(B)}}\sum_{i=1}^{s_{i'_1}^{(A)}} u_{i_1', i}^{(A)} \delta_{i, j}^{A, B}  & = & \sum_{i=1}^{s_{i_1'}^{(A)}} (\frac{u_{i_1', i}^{(A)}}{\tilde{u}_{i_1', i}^{(A)} })
( \tilde{u}_{i_1', i}^{(A)} m_{i, j}^{A, B}\frac{1}{\tilde{u}_{i_1'', j}^{(B)}}) 
 (\frac{\delta_{i, j}^{A, B}}{m_{i, j}^{A, B}}) \\
& <  & \sum_{i=1}^{s_{i_1'}^{(A)}} (\frac{u_{i_1', i}^{(A)}}{\tilde{u}_{i_1', i}^{(A)} })
( \tilde{u}_{i_1', i}^{(A)} m_{i, j}^{A, B}\frac{1}{\tilde{u}_{i_1'', j}^{(B)}}) 
 (1-\delta_1') \\  
 & = & (1- \delta_1')\sum_{i=1}^{s_{i_1'}^{(A)}} (\frac{u_{i_1', i}^{(A)}}{\tilde{u}_{i_1', i}^{(A)} })
( \tilde{u}_{i_1', i}^{(A)} m_{i, j}^{A, B}\frac{1}{\tilde{u}_{i_1'', j}^{(B)}}) \\
& = & (1- \delta_1') \frac{2}{\mathrm{dim}(X)}((\kappa_{i_1', i_1''}^{A, B})^* (r_{i_1''}^{(A)}))_j \\
& < & \frac{2}{\mathrm{dim}(X)} (r_{i_1''}^{(0)}{(B)})_j =  \frac{u_{i_1'', j}^{(B)}}{\tilde{u}_{i_1'', j}^{(B)}}.
\end{eqnarray*}
Therefore,
\begin{equation}\label{room-4-dim}
\sum_{i=1}^{s_{i_1'}^{(A)}} u^{(A)}_{i_1', i} \delta_{i, j}^{A, B} < u_{i_1'', j}^{(B)},\quad j=1, ..., s_{i_1''}^{(B)}.
\end{equation}

By \eqref{rapid-growth-cond-2}, $i_1''$ can be chosen even farther out so that  
%for any $i \geq i_1''$, 
%\begin{equation}\label{close-pert-1}
%([\psi_{i''_1, i}] - D[\psi_{i_1'', i}])(\mathbf 1_{s_{i_1''}^{(B)}}) < \frac{\delta_1'}{M_1}[\psi_{i_1'', i}](\mathbf 1_{s_{i_1''}^{(B)}}), \quad i > i_1''.  
%\end{equation}
%
%and
\begin{equation*}
\sum_{i=1}^{s^{(B)}_{i_1'}}([\psi_{i_1'', k}]_{i, j} -  (D[\psi_{i''_1, k}])_{i, j})\tilde{u}^{(B)}_{i_2', i} < \delta_1' \sum_{i=1}^{s^{(B)}_{i_1'}} [\psi_{i_1'', k}]_{i, j}\tilde{u}^{(B)}_{i_1', i} , \quad j=1, ..., s_{k}^{(B)},\ k> i_1'',
\end{equation*}
which implies
\begin{equation}\label{gr-con-B}
\sum_{i=1}^{s_{i_1''}^{B}}(D[\psi_{i''_1, k}])_{i, j})\tilde{u}^{(B)}_{i_2', i} > (1-\delta_1') \sum_{i=1}^{s_{i_1''}^{B}}  [\psi_{i_1'', k}]_{i, j}\tilde{u}^{(B)}_{i_1', i} , \quad j=1, ..., s_{k}^{(B)},\ k> i_1''.
\end{equation}

Then, define the map $$\rho'_1: A_{i_1'} \to B_{i_1''}$$ of the form  
\begin{eqnarray*}
\rho'_1(f_1, ..., f_{s^{(A)}_{i_1'}}) & \mapsto &  ( \mathrm{diag}\{f_1\circ P^{A, B}_{1, 1}, ..., f_{s_{i_1''}^{(A)}} \circ P^{A, B}_{s_{i_1'}^{(A)}, 1}, \textrm{point evaluations}\}, ..., \\
& & \mathrm{diag}\{f_1\circ P^{A, B}_{1, s_{i_1''}^{(B)}}, ..., f_{s_{i_1'}^{(A)}} \circ P^{A, B}_{s_{i_1'}^{(A)}, s_{i_1''}^{(B)}}, \textrm{point evaluations}\}),
\end{eqnarray*}
where each map 
$$P^{A, B}_{i, j}: X^{u^{(B)}_{i_1'', j}} \to X^{u^{(A)}_{i_1', i}},\quad i=1, ..., s^{(A)}_{i_1'},\ j=1, ..., s^{(B)}_{i_1''}, $$
consists of $ \delta^{A, B}_{i, j}$ disjoint coordinate projections, and the $\Kzero$-multiplicity of $\rho'_1$ is $\kappa_{i_1', i_1''}^{[A, B]}.$ Note that this is well defined by \eqref{room-4-dim}. Then, define $$\rho_1 = \rho_1' \circ \phi_{1, i_1'}: A_1 \to B_{i_1''}.$$
The construction above can be illustrated by the following diagram:
\begin{displaymath}
\xymatrix{
A_1 \ar[r]^{\phi_{1, i_1'}} & A_{i'_1} \ar[r] \ar[dr]^{\rho_1'} & A_{i_1''} \ar[r] & \cdots \ar[r] & A \\
B_1 \ar[r]  & B_{i'_1} \ar[r] &  B_{i_1''} \ar[r] & \cdots \ar[r] & B.
}
\end{displaymath}
%Note that the map $\rho_1$ also has the form

Let us now construct the map $\eta$ from $B$ back to $A$. By the construction of $B$, there is $\Delta_2>0$ such that for each $i \geq i_1''$ and if one writes $$\psi_{i_1'', i} = P \oplus \Theta: B_{i_1''} \to B_i,$$ where $P: B_{i_1''} \to B_i$ consists of coordinate projections and $\Theta: B_{i_1''} \to B_i$ consists of point evaluations, then $$\mathrm{tr}(\Theta(1_{B_{i_1''}})) > \Delta_2,\quad \tau \in \tr(B_i).$$ 
Then choose $\delta_2'$ sufficiently small that $$\frac{2\delta_2'  + (1-(1- 2 \delta_2')^3)}{\Delta_2} < \delta_2. $$ One also should ensure that $\delta_2' < \delta_1'$.

%Set $$\delta_2' = ? \quad \mathrm{and} \quad \delta''_2 = \min\{\delta_1', \delta_2'\}. $$
By the same argument as for $\rho_1$, there is a sufficiently large $i_2' > i_1''$ that
\begin{equation}\label{close-ratio-back}
(1-\delta''_2)(\kappa^{B, A}_{i, j})^*(r_i^{(0)}{(B)}) < r_j^{(0)}{(A)},\quad i, j \geq i_2',
\end{equation}
and
\begin{equation}\label{gr-con-B-step-2}
\sum_{i=1}^{s^{(B)}_{i_2'}}([\psi_{i_2', k}]_{i, j} -  (D[\psi_{i'_2, k}])_{i, j})\tilde{u}^{(B)}_{i_1', i} < \delta_2' \sum_{i=1}^{s^{(B)}_{i_2'}} [\psi_{i_2', k}]_{i, j}\tilde{u}^{(B)}_{i_2', i} , \quad j=1, ..., s_{k}^{(B)},\ k> i_2',
\end{equation}
and then there is $i_2'' > i_2'$ sufficiently large that 
\begin{equation}\label{gr-con-A-step-2}
\sum_{i=1}^{s_{i_2''}^{A}}(D[\psi_{i''_2, k}])_{i, j})\tilde{u}^{(A)}_{i_2'', i} > (1-\delta_2') \sum_{i=1}^{s_{i_2''}^{A}}  [\psi_{i_2'', k}]_{i, j}\tilde{u}^{(A)}_{i_2'', i} , \quad j=1, ..., s_{k}^{(A)},\ k> i_2'',
\end{equation}
and there are positive integers
$$\delta_{i, j}^{B, A}, \quad i=1, ..., s_{i_2'}^{(B)},\ j=1, ..., s_{i_2''}^{(A)}, $$
satisfying
\begin{equation}\label{app-multi-B}
1-2\delta'_2<\frac{\delta_{i, j}^{B, A}}{m_{i, j}^{B, A}} < 1 - \delta'_2 ,\quad i=1, ..., s_{i_2'}^{(B)},\ j=1, ..., s_{i_2''}^{(A)},
\end{equation}
where $(m_{i, j}^{B, A})$, $i=1, ..., s_{i_2'}^{(B)}$, $j=1, ..., s_{i_2''}^{(A)}$, are the multiplicities of $\kappa_{i_2', i_2''}^{B, A}$,
and therefore
\begin{equation}\label{room-4-dim-back}
\sum_{i=1}^{s_{i_2'}^{(B)}} u^{(B)}_{i_1', i} \delta_{i, j}^{B, A} < u_{i_2'', j}^{(A)},\quad j=1, ..., s_{i_2''}^{(A)}.
\end{equation}
Thus, there is room to define a map $\eta'_1: B_{i_2'} \to A_{i_2''}$ with the multiplicities of coordinate projections equal to $(\delta_{i, j}^{B, A})$. Then  define $$\eta_1:= \eta'_1 \circ \psi_{i_1'', i_1'}. $$ The construction can be illustrated by the following diagram:  
\begin{displaymath}
\xymatrix{
A_1 \ar[r]^{\phi_{1, i_1'}} & A_{i'_1} \ar[r] \ar[dr]^{\rho_1'} & A_{i_1''} \ar[r] & A_{i_2'} \ar[r] & A_{i_2''} \ar[r] &  \cdots \ar[r] & A \\
B_1 \ar[r]  & B_{i'_1} \ar[r] &  B_{i_1''} \ar[r]_{\psi_{i_1'', i_2'}} & B_{i_2'} \ar[ur]^{\eta_1'} \ar[r] & B_{i_2''} \ar[r] & \cdots \ar[r] & B.
}
\end{displaymath}

Let us consider the composition $\eta_1 \circ \rho_1$, which is $(\eta_1' \circ \psi_{i_1'', i_2'} \circ \rho'_1) \circ \phi_{1, i_1'}$, and compare it with the map $\phi_{1, i_2''}$.

Note that the multiplicity matrix of coordinate projections of $\eta_1' \circ \psi_{i_1'', i_2'} \circ \rho'_1$ is the product $$(\delta_{i, j}^{B, A})(D[\psi_{i_1'', i'_2}])(\delta_{i, j}^{A, B}).$$ Then, using \eqref{app-multi-A}, \eqref{gr-con-B}, and  \eqref{app-multi-B} (note that $\delta_2'<\delta_1'$), one has 
\begin{eqnarray*}
(\delta_{i, j}^{B, A})(D[\psi_{i_1'', i'_2}])(\delta_{i, j}^{A, B})(\tilde{u}^{(A)}_{i_1', i})
& > & (1-2\delta_1')(m_{i, j}^{B, A})(D[\psi_{i_1'', i'_2}])(1-2\delta_2'')(m_{i, j}^{A, B})(\tilde{u}^{(A)}_{i_1', i}) \\
& > &(1-2\delta_1')^2 (m_{i, j}^{B, A})((D[\psi_{i_1'', i'_2}])(m_{i, j}^{A, B})) (\tilde{u}^{(A)}_{i_1', i}) \\
& = &(1-2\delta_1')^2 (m_{i, j}^{B, A}) (D[\psi_{i_1'', i'_2}]) (\tilde{u}^{(B)}_{i_1'', i}) \\
& > &(1-2\delta_1')^3 (m_{i, j}^{B, A})[\psi_{i_1'', i'_2}](\tilde{u}^{(B)}_{i_1'', i}) \\
& = & (1-2\delta_1')^3(\kappa_{i_2', i_2''}^{B, A} \circ [\psi_{i_1'', i_2'}] \circ \kappa^{A, B}_{i_1', i_1''}) (\tilde{u}^{(A)}_{i_1', i}) \\
& = & (1-2\delta_1')^3  [\phi_{i_1', i_2''}] (\tilde{u}^{(A)}_{i_1', i}),
%& > & (1-2\delta_1')^3  D[\phi_{i_1', i_2''}].
\end{eqnarray*}
and therefore
$$ (1-2\delta_1')^3 [\phi_{i_1', i_2''}](\tilde{u}^{(A)}_{i_1', i}) < D[\eta_1' \circ \psi_{i_1'', i_2'} \circ \rho'_1] (\tilde{u}^{(A)}_{i_1', i}) \leq [\phi_{i_1', i_2''}] (\tilde{u}^{(A)}_{i_1', i}). $$
That is, 
%\begin{equation}\label{intertwin-est-1}
%\abs{ [\phi_{i_1', i_2''}](\tilde{u}^{(A)}_{i_1', i}) - D[\eta_1' \circ \psi_{i_1'', i_2'} \circ \rho'_1](\tilde{u}^{(A)}_{i_1', i}) }< (1 - (1-2\delta'_1)^3) [\phi_{i'_1, i_2''}] (\tilde{u}^{(A)}_{i_1', i}).
%\end{equation}
\begin{equation}\label{intertwin-est-1}
\sum_{i=1}^{s_{i_1'}^{(A)}}( [\phi_{i_1', i_2''}]_{i, j} \tilde{u}^{(A)}_{i_1', i} - D[\eta_1' \circ \psi_{i_1'', i_2'} \circ \rho'_1]_{i, j} \tilde{u}^{(A)}_{i_1', i} )  < (1 - (1-2\delta'_1)^3)  \sum_{i=1}^{s_{i_1'}^{(A)}}  [\phi_{i'_1, i_2''}]_{i, j} (\tilde{u}^{(A)}_{i_1', i}), 
\end{equation}
for each $j=1, ..., s_{i_2''}^{(A)}$.

Also note that, by \eqref{gr-con-A}, 
%\begin{equation}
%\sum_{i=1}^{s^{(A)}_{i_1'}}([\phi_{i_1', i_2''}]_{i, j} -  (D[\phi_{i'_1, i_2''}])_{i, j}) < \delta_1' \sum_{i=1}^{s^{(A)}_{i_1'}} [\phi_{i_1', i_2''}]_{i, j}, \quad j=1, ..., s_{i_2''}^{(A)}. 
%\end{equation}
%or (which is more reasonable)
\begin{equation}\label{small-eva}
\sum_{i=1}^{s^{(A)}_{i_1'}}([\phi_{i_1', i_2''}]_{i, j} -  (D[\phi_{i'_1, i_2''}])_{i, j})\tilde{u}^{(A)}_{i_1', i} < \delta_1' \sum_{i=1}^{s^{(A)}_{i_1'}} [\phi_{i_1', i_2''}]_{i, j}\tilde{u}^{(A)}_{i_1', i} = \delta_1'\tilde{u}^{(A)}_{i_2'', j} , \quad j=1, ..., s_{i_2''}^{(A)}. 
\end{equation}

Therefore, for each $j=1, ..., s_{i_2''}^{(A)}$, using the equation 
$$\sum_{i=1}^{s^{(A)}_{i_1'}} [\phi_{i_1', i_2''}]_{i, j} \tilde{u}_{i_1', i}^{(A)} = u_{i_2'', j}^{(A)}$$
in the last step, one has
\begin{eqnarray}
 & &\sum_{i=1}^{s^{(A)}_{i_1'}} \abs{ (D[\eta_1' \circ \psi_{i_1'', i_2'} \circ \rho'_1])_{i, j}\tilde{u}_{i_1', i}^{(A)} - (D[\phi_{i'_1, i_2''}])_{i, j} \tilde{u}_{i_1', i}^{(A)} } \label{dim-almost-commut-0} \\
 & \leq & \sum_{i=1}^{s^{(A)}_{i_1'}}( \abs{ (D[\eta_1' \circ \psi_{i_1'', i_2'} \circ \rho'_1])_{i, j}\tilde{u}_{i_1', i}^{(A)} - [\phi_{i_1', i_2''}]_{i, j}\tilde{u}_{i_1', i}^{(A)} } +   \abs{ [\phi_{i_1', i_2''}]_{i, j}\tilde{u}_{i_1', i}^{(A)}  - (D[\phi_{i'_1, i_2''}])_{i, j}\tilde{u}_{i_1', i}^{(A)} } ) \nonumber \\
 & < &  \delta_1' (\sum_{i=1}^{s^{(A)}_{i_1'}} [\phi_{i_1', i_2''}]_{i, j} \tilde{u}_{i_1', i}^{(A)}) + (1-(1- 2 \delta_1')^3) (\sum_{i=1}^{s^{(A)}_{i_1'}} [\phi_{i_1', i_2''}]_{i, j} \tilde{u}_{i_1', i}^{(A)}) \nonumber \\
 & = & (\delta_1'  + (1-(1- 2 \delta_1')^3) ) \tilde{u}_{i_2'', j}^{(A)}. \nonumber
\end{eqnarray}
%
%(The desired estimation should be
%$$ 
%\sum_{i=1}^{s^{(A)}_{i_1'}} \abs{ (D[\eta_1' \circ \psi_{i_1'', i_2'} \circ \rho'_1])_{i, j} - (D[\phi_{i'_1, i_2''}])_{i, j} }\tilde{u}_{i_1', i}^{(A)} <  (\delta_1'  + (1-(1- 2 \delta_1')^3) ) \sum_{i=1}^{s^{(A)}_{i_1'}} [\phi_{i_1', i_2''}]_{i, j} \tilde{u}_{i_1', i}^{(A)},\quad \forall j.
%$$
%Note that
%$$\sum_{i=1}^{s^{(A)}_{i_1'}} [\phi_{i_1', i_2''}]_{i, j} \tilde{u}_{i_1', i}^{(A)} = u_{i_2'', j}^{(A)},\quad j=1, ..., s_{i_2''}^{(A)}.$$
%)
Then, introducing the matrix $(c_{i, j})$ with 
$$c_{i, j} := \min\{(D[\eta_1' \circ \psi_{i_1'', i_2'} \circ \rho'_1])_{i, j}, [\phi_{i_1', i_2''}]_{i, j}\}, i=1,..., s_{i_1'}^{(A)},\ j=1,..., s_{i_2''}^{(A)},$$
and defining $P: A_{i_1'} \to A_{i_2''}$ to be the diagonal map consisting of coordinate projections with multiplicities given by $(c_{i, j})$, one has the decompositions 
$$ \phi_{i_1', i_2''} = P \oplus R_0,$$
$$ \eta_1' \circ \psi_{i_1'', i_2'} \circ \rho'_1 = P \oplus R_1,  $$
where, by \eqref{dim-almost-commut-0} and \eqref{small-eva}, for each $j=1, ..., s_{i''_2}^{(A)}$, 
\begin{equation}\label{dim-almost-commut}
\mathrm{rank}_j((R_0(1_{A_{i_1'}})))= \mathrm{rank}_j((R_1(1_{A_{i_1'}}))) < ((\delta_1'  + (1-(1- 2 \delta_1')^3) ) + \delta_1')\tilde{u}^{(A)}_{i_2'', j}.
\end{equation}
Write 
$$\phi_{1, i_1'} = P_1 \oplus \Theta,$$ 
where $P_1: A_1 \to A_{i_1'}$ is a coordinate projection map and $\Theta: A_1 \to A_{i_1'}$ is a point evaluation map;
then
\begin{eqnarray*}
\phi_{1, i_2''} & = & \phi_{i_1', i_2''} \circ \phi_{1, i_1'} \\
& = & (P \oplus R_0) \circ (P_1 \oplus \Theta)\\
& = & (P \circ P_1) \oplus (R_0 \circ P_1) \oplus ((P \oplus R_0) \circ \Theta)
\end{eqnarray*}
and
\begin{eqnarray*}
\eta_1 \circ  \rho_1 & = & (\eta_1' \circ \psi_{i_1'', i_2'} \circ \rho'_1) \circ \phi_{1, i_1'} \\
& = & (P \oplus R_1) \circ (P_1 \oplus \Theta)\\
& = & (P \circ P_1) \oplus (R_1 \circ P_1) \oplus ((P \oplus R_1) \circ \Theta).
\end{eqnarray*}

Note that, since $\Theta$ is a point-evaluation map, 
\begin{equation}\label{same-pe}
(P\oplus R_0)\circ \Theta = (P\oplus R_1) \circ \Theta. 
\end{equation}

By \eqref{dim-almost-commut} and  \eqref{lower-bd}, for each $j=1, ..., s_{i_2''}^{(A)}$, one has
\begin{eqnarray*}
\frac{\mathrm{rank}_j(R_0 \circ P_1)(1_{A_1})}{\mathrm{rank}_j((P \oplus R_0) \circ \Theta)(1_{A_1})} 
 & \leq & \frac{\mathrm{rank}_j R_0(1_{A_{1_i'}})}{\mathrm{rank}_j(\phi_{i_1', i_2''}  \circ \Theta)(1_{A_1})} \\
& \leq & \frac{ ((\delta_1'  + (1-(1- 2 \delta_1')^3) ) + \delta_1')\tilde{u}^{(A)}_{i_2'', j} }{\mathrm{rank}_j(\phi_{i_1', i_2''} \circ \Theta)(1_{A_1})} \\
& = & \frac{ (\delta_1'  + (1-(1- 2 \delta_1')^3) ) + \delta_1' }{(\mathrm{tr}_j\circ \phi_{i_1', i_2''})(\Theta(1_{A_1}))} \\
& < & \frac{ (\delta_1'  + (1-(1- 2 \delta_1')^3) ) + \delta_1' }{\Delta_1} < \delta_1,
\end{eqnarray*}
and, by the same argument,
$$\frac{\mathrm{rank}_j(R_1 \circ P_1)(1_{A_1})}{\mathrm{rank}_j((P \oplus R_1) \circ \Theta)(1_{A_1})} < \delta_1,\quad j=1, ..., s_{i_2''}^{(A)}. $$
Then the maps $\rho_1$ and $\eta_1$ possess the properties of the proposition with $$P_{A, 1} = P \circ P_1,\quad  R'_{A, 1} = R_0\circ P_1,\quad R''_{A, 1} = R_0 \circ P_1, \quad \mathrm{and} \quad \Theta_{A, 1} = (P \oplus R_0) \circ \Theta.$$

Repeating this process, one has the maps $\rho_i, \eta_i$, $i=1, 2,...$, which have the desired property.
\end{proof}

Recall the following stable uniqueness theorem:
\begin{thm}[Theorem 7.5 of \cite{ELN-Vill}]\label{stable-uniq}
Let $X$ be a K-contractible metrizable compact space (i.e., $\Kzero(\mathrm{C}(X)) = \Int$ and $\Kone(\mathrm{C}(X)) = \{0\}$), and let $\Delta: \mathrm{C}(X)^+ \to (0, +\infty)$ be a map. For any finite set $\mathcal F\subseteq \mathrm{C}(X)$ and any $\eps>0$, there exists a finite set $\mathcal H\subseteq \mathrm{C}(X)^+$ with $\mathrm{supp}(h) \neq X$ for each $h\in\mathcal H$ and there exists $M\in \mathbb N$ such that the following property holds:
for any unital homomorphisms $$\phi, \psi: \mathrm{C}(X) \to \mathrm{M}_{n}(\mathrm{C}(Y))\quad \textrm{and}\quad \theta: \mathrm{C}(X) \to \mathrm{M}_{m}(\Comp) \subseteq \mathrm{M}_{m}(\mathrm{C}(Y)),$$
where $\theta$ is a unital point-evaluation map with $nM < m,$ and such that
$$\mathrm{tr}(\theta(h)) > \Delta(h),\quad h\in\mathcal H,$$
there is a unitary $u\in\mathrm{M}_{n+m}(\mathrm{C}(Y))$ such that
$$\norm{\mathrm{diag}\{\phi(a), \theta(a)\} - u^*\mathrm{diag}\{\psi(a), \theta(a)\} u} < \eps,\quad a\in\mathcal F.$$ 
\end{thm}

\begin{thm}\label{classification-AF}
Let $X$ be a K-contractible solid space such that $ 0 < \mathrm{dim}(X) < \infty$, and let $A(X, G, \mathcal E)$ and $B(X, H, \mathcal F)$ be AF-Villadsen algebras with seed space $X$ satisfying \eqref{rapid-growth-cond} (and therefore \eqref{rapid-growth-cond-2}), where $G$ and $H$ are Bratteli diagrams and $\mathcal E$ and $\mathcal F$ are point evaluation sets. Then $A \cong B$ if, and only if, $(\mathrm{Cu}(A), [1_A]) \cong (\mathrm{Cu}(B), [1_B])$. Indeed, $A \cong B$ if, and only if, $$((\Kzero(A), \Kzero^+(A), [1_A]_0), r_\infty^{(0)}{(A)}) \cong ((\Kzero(B), \Kzero^+(B), [1_B]_0), r_\infty^{(0)}{(B)}).$$
\end{thm}

\begin{proof}
Assume that $$(\mathrm{Cu}(A), [1_A]) \cong (\mathrm{Cu}(B), [1_B]).$$ By Theorem \ref{comparison-property}, this implies that $$((\Kzero(A), \Kzero^+(A), [1_A]_0), r_\infty^{(0)}{(A)}) \cong ((\Kzero(B), \Kzero^+(B), [1_B]_0), r_\infty^{(0)}{(B)}).$$
%(Note that, by \eqref{rapid-growth-cond}, the functions $r_\infty^{(A)}$ and $r_\infty^{(B)}$ are non-zero everywhere. By \eqref{rapid-growth-cond-2}, the trace simplex of $A$ or $B$ is affinely homeomorphic to the trace simplex of the algebra before adding point evaluation. By Proposition \ref{iso-non-simple}, the algebra before adding point evaluations are isomorphic. Hence the trace simplexes of $A$ and $B$ are affinely homeomorphic.)

Let us prove the implication  
$$((\Kzero(A), \Kzero^+(A), [1_A]_0), r_\infty^{(0)}{(A)}) \cong ((\Kzero(B), \Kzero^+(B), [1_B]_0), r_\infty^{(0)}{(B)}) \quad \Longrightarrow \quad A \cong B.$$

Choose finite subsets $\mathcal F^{(A)}_i \subseteq A_i$, $\mathcal F^{(B)}_i \subseteq A_i$, $i=1, 2, ...$, such that $$\mathcal F^{(A)}_1 \subseteq \mathcal F^{(A)}_2 \subseteq \cdots \quad \mathrm{and} \quad  \mathcal F^{(B)}_1 \subseteq \mathcal F^{(B)}_2 \subseteq \cdots,$$
and 
$$ \overline{\bigcup_{i=1}^\infty \mathcal F^{(A)}_i} = A \quad \mathrm{and} \quad \overline{\bigcup_{i=1}^\infty \mathcal F^{(B)}_i} = B.$$
Choose $\eps_1>\eps_2> \cdots >0$ such that $$\sum_{i=1}^\infty \eps_i < \infty.$$

For each $A_i$, $i=1, 2, ...$, consider $$\Delta^{(A)}(a):= \inf\{\tau(a): \tau \in \tr(A)\},\quad a \in A_i^+, $$
and for each $B_i$, $i=1, 2, ...$,
$$\Delta^{(B)}(b):= \inf\{\tau(b): \tau \in \tr(B)\},\quad b \in B_i^+.$$

For each $(\mathcal F_i^{(A)}, \eps_i)$, applying  Theorem \ref{stable-uniq} with respect to $\frac{1}{2}\Delta^{(A)}$, one obtains a finite set of positive contractions $\mathcal H_i^{(A)} \subseteq A_i$  and $M_i^{(A)} \in \mathbb N$ with the property of Theorem \ref{stable-uniq}. 
Similarly, for each $(\mathcal F_i^{(B)}, \eps_i)$, applying  Theorem \ref{stable-uniq} with respect to $\frac{1}{2}\Delta^{(B)}$, one obtains a finite set of positive contractions $\mathcal H_i^{(B)} \subseteq B_i$ and $M_i^{(B)} \in \mathbb N$ with the property of Theorem \ref{stable-uniq}. 

By Proposition \ref{pre-intertwining}, upon a telescoping, there is a diagram
\begin{equation}
\xymatrix{
A_1 \ar[r]^{\phi_1} \ar[d]_{\rho_1} & A_2 \ar[r]^{\phi_2} \ar[d]_{\rho_2} & \cdots & \\
B_1 \ar[r]^{\psi_1} \ar[ur]^{\eta_1} & B_2 \ar[r]^{\psi_1} & \cdots & 
}
\end{equation}
such that each of $\phi_i, \psi_i, \rho_i, \kappa_i$, $i=1, 2, ...$, consists of independent coordinate projections and point evaluations, i.e., restricted to each direct summand of the domain, it has the form  
$$ (f_1, f_2, ..., f_s) \mapsto \mathrm{diag}\{ f_1 \circ P_1, ..., f_s \circ P_s, \textrm{point evaluations}\},$$
where $P_1, ..., P_s$ are mutually disjoint sets of coordinate projections,
and
for each $i=1, 2, ...$, there are decompositions
$$\phi_i = \mathrm{diag}\{P_{A, i}, R'_{A, i}, \Theta_{A, i}\}, \quad \psi_i = \mathrm{diag}\{P_{B, i}, R'_{B, i}, \Theta_{B, i}\},  $$
and
$$ \eta_i\circ\rho_i =  \mathrm{diag}\{P_{A, i}, R''_{A, i}, \Theta_{A, i}\},\quad  \rho_{i+1}\circ\eta_i =  \mathrm{diag}\{P_{B, i}, R''_{B, i}, \Theta_{B, i}\},$$
where $P_{A, i}: A_i \to A_{i+1}$ and $P_{B, i}: B_i \to B_{i+1}$ consist of coordinate projections, and $\Theta_{A, i}: A_i \to A_{i+1}$ and $\Theta_{B, i}: B_i \to B_{i+1}$ consist of point evaluations, 
such that, for each $i=1, 2, ...$, 
$$ \frac{\mathrm{rank}_j(R'_{A, i}(1_{A_i}))}{\mathrm{rank}_j(\Theta_{A, i}(1_{A_i}))} = \frac{\mathrm{rank}_j(R''_{A, i}(1_{A_i}))}{\mathrm{rank}_j(\Theta_{A, i}(1_{A_i}))} < \delta_i,\quad j=1, ..., s_{i+1}^{(A)},
$$
and
$$
\frac{\mathrm{rank}_j(R'_{B, i}(1_{B_i}))}{\mathrm{rank}_j(\Theta_{B, i}(1_{B_i}))} = \frac{\mathrm{rank}_j(R''_{B, i}(1_{B_i}))}{\mathrm{rank}_j(\Theta_{B, i}(1_{B_i}))} < \delta_i, \quad j=1, ..., s_{i+1}^{(B)},
 $$
 where 
 $$ \delta_i = \min\{\frac{1}{M_i^{(A)}},\ \frac{1}{M_i^{(B)}},\ \frac{1}{2}\Delta^{(A)}(h^{(A)}),\ \frac{1}{2}\Delta^{(B)}(h^{(B)}): h^{(A)} \in \mathcal H_i^{(A)},\ h^{(B)} \in \mathcal H_i^{(B)}  \}. $$

Let us compare the maps 
$$ \phi_1 =  \mathrm{diag}\{P_{A, 1}, R'_{A, 1}, \Theta_{A, 1}\} \quad \mathrm{and} \quad \eta_1 \circ\rho_1 = \mathrm{diag}\{P_{A, 1}, R''_{A, 1}, \Theta_{A, 1}\}.$$ Since none of the elements of $\mathcal H_1^{(A)}$ has full support, for each $h \in \mathcal H_1^{(A)}$, there is $x_0 \in X_1$, where $X_1$ is the base space of $A_1$, such that $h(x_0) = 0$. Since the map $P_{A, 1}$ consists of coordinate projections, there is $y_0 \in X_2$, where $X_2$ is the base space of $A_2$, such that $P_{A, 1}(h)(y_0) = 0$, and therefore
$$\mathrm{tr}(\phi_1(h)(y_0)) =  \mathrm{tr}(\mathrm{diag}\{P_{A, 1}(h)(y_0), R'_{A, 1}(h)(y_0), \Theta_{A, 1}(h)(y_0)\}) > \Delta_1^{(A)}(h),$$
where $\mathrm{tr}$ is the normalized trace of the matrix algebra over $y_0$.

Hence
$$\frac{\mathrm{Tr}(\phi_1(h)(y_0))}{\mathrm{rank}(1_{R_{A, 1}} + 1_{\Theta_{A, 1}})} = \frac{\mathrm{Tr}(R'_{A, 1}(h)(y_0)) + \mathrm{Tr}(\Theta_{A, 1}(h)(y_0)) }{\mathrm{rank}(R_{A, 1}(1_{A_1}) + \Theta_{A, 1}(1_{A_1}))} > \Delta_1^{(A)}(h),$$
where $\mathrm{Tr}$ is the unnormalized trace of the matrix algebra over $y_0$, and therefore (note that the image of $\Theta_{A, 1}$ consists of constant functions), 
\begin{eqnarray*}
\frac{\mathrm{Tr}(\Theta_{A, 1}(h))}{\mathrm{rank}( \Theta_{A, 1}(1_{A_1}) )} 
& > &  
\Delta_1^{(A)}(h) ( \frac{\mathrm{rank}(R_{A, 1}(1_{A_1}))}{\mathrm{rank}(\Theta_{A, 1}(1_{A_1}))} + 1 ) - \frac{\mathrm{Tr}(R'_{A, 1}(h)(y_0))}{\mathrm{rank}(\Theta_{A, 1}(1_{A_1}))} \\
& > & \Delta_1^{(A)}(h) ( \delta_1 + 1 ) -  \frac{\mathrm{rank}(R'_{A, 1}(1_{A_1}))}{\mathrm{rank}(\Theta_{A, 1}(1_{A_1}))} \cdot \frac{\mathrm{Tr}(R'_{A, 1}(h)(y_0))}{\mathrm{rank}(R'_{A, 1}(1_{A_1}))} \\
& > & \Delta_1^{(A)}(h) -  \delta_1 \\
& > & \frac{1}{2}\Delta_1^{(A)}(h).
\end{eqnarray*}

Since
$$ \frac{\mathrm{rank}_j(R'_{A, 1}(1_{A_1}))}{\mathrm{rank}_j(\Theta_{A, 1}(1_{A_1}))} = \frac{\mathrm{rank}_j(R''_{A, 1}(1_{A_1}))}{\mathrm{rank}_j(\Theta_{A, 1}(1_{A_1}))} < \delta_i < \frac{1}{M^{(A)}_1},\quad j=1, ..., s_{2}^{(A)},
$$
it follows from Theorem \ref{stable-uniq} that there is a unitary $u_1 \in A_2 $ such that
$$ \norm{\phi_1(f) - u^*_1(\eta_1\circ \rho_1)(f)u_1} <\eps_1,\quad f \in \mathcal F^{(A)}. $$ Replacing $\eta_1(\cdot)$ by $u_1^*\eta(\cdot)u_1$, and still denoting it by $\eta_1$, we have 
$$ \norm{\phi_1(f) - (\eta_1\circ \rho_1)(f)} <\eps_1,\quad f \in \mathcal F^{(A)}. $$

Repeating this process, we have a diagram 
\begin{equation}
\xymatrix{
A_1 \ar[r]^{\phi_1} \ar[d]_{\rho_1} & A_2 \ar[r]^{\phi_2} \ar[d]_{\rho_2} & \cdots & \\
B_1 \ar[r]^{\psi_1} \ar[ur]^{\eta_1} & B_2 \ar[r]^{\psi_1} & \cdots & 
}
\end{equation}
such that for each $i=1, 2, ...$,
$$\norm{\phi_i(f) - (\eta_i\circ\rho_i)(f)} < \eps_i,\quad f \in \mathcal F_i^{(A)}, $$
and
$$\norm{\psi_i(f) - (\rho_{i+1}\circ\eta_i)(f)} < \eps_i,\quad f \in \mathcal F_i^{(B)}.$$
By the approximate intertwining argument (Theorems 2.1 and 2.2 of \cite{Ell-Cre}), we have $A \cong B$, as desired.
\end{proof}

\begin{rem}
Note that, Theorem \ref{classification-AF} implies that the trace simplex is determined by the $\Kzero$-group. It would be interesting to see what it is. Is it independent of $X$, as the UHF-Villadsen algebra case (see Theorem 4.5 of \cite{ELN-Vill})?
\end{rem}

\section{A generalized version of the comparison radius function for UHF-Villadsen algebras}\label{general-gap-functions}

Note that the function $r^{(0)}_\infty$ of Section \ref{AF-construction} and Section  \ref{section-comparison-property} degenerates to the radius of comparison $\mathrm{rc}(A)$ if $A$ is a Villadsen algebra of UHF type. In this section, let us introduce a more general function, denoted by $r_\infty$, for Villadsen algebras of UHF type, which is not constant in general, but still has similar properties to the (numerical) radius of comparison (Theorem \ref{low-env-gn}).

%Recall the (continuous) gap functions:

\begin{defn}
Let $X$ be a compact Hausdorff space. For each $x \in X$, define
$$\mathrm{loc.dim}(x) = \min\{\mathrm{dim}(V): \textrm{$V$ is a closed neighbourhood of $x$} \}.$$

Note that the function $x \mapsto \mathrm{loc.dim}(x)$ is upper semicontinuous, and, if $X$ is a (finite) simplicial complex, then 
$$\mathrm{loc.dim}((x_1, ..., x_n)) = \mathrm{loc.dim}(x_1) + \cdots + \mathrm{loc.dim}(x_n),\quad (x_1, ..., x_n) \in X^n. $$
\end{defn}

Let $X$ be a (finite) simplical complex, and let $A(X, (n_s), (k_s))$ be a Villadsen algebra with seed space $X$ (see \cite{ELN-Vill} and \cite{Vill-perf}). Let us briefly recall its construction \cite{ELN-Vill}: $A(X, (n_s), (k_s))$ is the inductive limit of the sequence 
\begin{equation}\label{Vill-lim}
\xymatrix{
\mathrm{C}(X) \ar[r] & \mathrm{M}_{(n_1+k_1)}(\mathrm{C}(X^{n_1})) \ar[r] & \mathrm{M}_{(n_1+k_1)(n_2+k_2)}(\mathrm{C}(X^{n_1n_2})) \ar[r] & \cdots ,
}
\end{equation}
where the seed for the $s$th-stage map,
$$\phi_i: \mathrm{C}(X^{n_1 \cdots n_{s-1}}) \to \mathrm{M}_{n_s+k_s}(\mathrm{C}(X^{n_1 \cdots n_{s-1}n_{s}})),$$ is defined by
\begin{eqnarray*} 
f & \mapsto & \mathrm{diag}\{  f\circ \pi_1, ... ,  f\circ \pi_{n_s}, f(\theta_{s, 1}), ..., f(\theta_{s, k_s})\} 
\end{eqnarray*}
where $\theta_{s, 1}, ..., \theta_{s, k_s} \in X^{n_1\cdots n_{s-1}}$ are evaluation points. The evaluation points are chosen in such a way (dense enough) that the limit algebra is simple, and the growth sequences $(n_s)$ and $(k_s)$ are chosen so that 
\begin{equation}\label{rdg-cond}
 \lim_{s\to\infty} \lim_{t\to\infty} \frac{n_s \cdots n_{t}}{(n_s + k_s) \cdots (n_{t} + k_{t})} = \lim_{i\to\infty}\lim_{j\to\infty} (\frac{n_s}{n_s+k_s}) \cdots (\frac{n_{t}}{n_{t}+k_{t}}) = 1.
 \end{equation}
Note that, by Corollary 6.2 of \cite{ELN-Vill}, the simple limit algebra is independent of the evaluation points.

For each $s=1, 2, ...$, consider the function
$$r_s(x) =\frac{1}{2} \cdot \frac{\mathrm{loc.dim}(x)}{(n_1 + k_1) \cdots (n_{s-1} + k_{s-1})} , \quad x \in X^{n_1\cdots n_{s-1}}. $$
It is an upper semicontinuous function on $X^{n_1 \cdots n_{s-1}} = \partial \tr(A_s)$, and hence is an upper semicontinuous affine function on $\tr(A_s)$, and hence on $\tr(A)$.

On regarding $x \mapsto \mathrm{loc.dim}(x)$ as the upper left corner of $A_s$, there is a decreasing sequence of positive contractions $(f_{s, n}) \subseteq A_s$ such that
\begin{equation}\label{loc-sequence}
\lim_{n\to\infty}\tau(f_{s, n}) = r_s(\tau),\quad \tau \in \tr(A_s).
\end{equation}
(This will be used in the proof of Theorem \ref{low-env-gn}(1) below.)

\begin{lem}\label{uniform-dec-r}
The sequence $r_1, r_2, ...$, of upper semicontinuous positive real-valued affine functions on $\tr(A)$ is decreasing, and for any $s < t$,
$$ \norm{r_s - r_t}_\infty \leq \frac{1}{2} \mathrm{dim}(X) (\frac{n_1 \cdots n_{s-1}}{(n_1 + k_1) \cdots (n_{s-1} + k_{s-1})} - \frac{n_1 \cdots n_{t-1}}{(n_1 + k_1) \cdots (n_{t-1} + k_{t-1})}).$$
 %$$\norm{r_s - r_{s+1}}_\infty \leq  \frac{n_1\cdots n_{s-1} k_s}{ (n_1 + k_1) \cdots (n_s + k_s) } \mathrm{dim}(X). $$
\end{lem}
\begin{proof}
For each $x \in X^{(n_1\cdots n_{s-1})n_s}$, note that
\begin{eqnarray*}
 (\varphi_{s})_*(r_s)((x_1, ..., x_{n_s})) & = & \frac{1}{n_s + k_s}( r_s(x_1) + \cdots + r_s(x_{n_s}) + r_s(\theta_1) + \cdots + r_s(\theta_{k_s}) ) \\
 & = & \frac{1}{2} \cdot \frac{1}{(n_1 + k_1) \cdots (n_s + k_s)}( \mathrm{loc.dim}(x_1) + \cdots + \mathrm{loc.dim}(x_{n_s})) + \\
 &&  \frac{1}{2} \cdot  \frac{1}{(n_1 + k_1) \cdots (n_s + k_s)} (\mathrm{loc.dim}(\theta_1) + \cdots +  \mathrm{loc.dim}(\theta_{k_s})) \\
 & = & \frac{1}{2} \cdot \frac{1}{(n_1 + k_1) \cdots (n_s + k_s)} \mathrm{loc.dim}(x_1, ..., x_{n_s}) +  \\
 &&  \frac{1}{2} \cdot \frac{1}{(n_1 + k_1) \cdots (n_s + k_s)} (\mathrm{loc.dim}(\theta_1) + \cdots +  \mathrm{loc.dim}(\theta_{k_s})).
 \end{eqnarray*}
In particular, $r_s > r_{s+1}$ and
\begin{eqnarray*}
 \norm{ r_s - r_{s+1}}_\infty & = & \norm{  \frac{1}{2} \cdot \frac{1}{(n_1 + k_1) \cdots (n_s + k_s)} (\mathrm{loc.dim}(\theta_1) + \cdots +  \mathrm{loc.dim}(\theta_{k_s})) }_\infty \\ 
 & \leq & \frac{1}{2} \cdot \frac{k_s}{ (n_1 + k_1) \cdots (n_s + k_s) } (n_1\cdots n_{s-1}) \mathrm{dim}(X).
\end{eqnarray*} 
%as desired.

In general, the same argument shows that for any $s < t$, 
\begin{eqnarray*}
 \norm{r_s - r_t}_\infty & \leq & \frac{1}{2} \cdot \frac{(n_s + k_s)\cdots (n_{t-1} + k_{t-1}) - n_s\cdots n_{t-1}}{(n_1 + k_1) \cdots (n_{t-1} + k_{t-1})} (n_1\cdots n_{s-1}) \mathrm{dim}(X) \\
 & = &  \frac{1}{2} \cdot \frac{n_1\cdots n_{s-1} (n_s + k_s)\cdots (n_{t-1} + k_{t-1}) - n_1 \cdots n_{t-1}}{(n_1 + k_1) \cdots (n_{t-1} + k_{t-1})} \mathrm{dim}(X) \\
 & = &  \frac{1}{2} \cdot (\frac{n_1 \cdots n_{s-1}}{(n_1 + k_1) \cdots (n_{s-1} + k_{s-1})} - \frac{n_1 \cdots n_{t-1}}{(n_1 + k_1) \cdots (n_{t-1} + k_{t-1})})  \mathrm{dim}(X).
 \end{eqnarray*}
\end{proof}

Thus, by \eqref{rdg-cond}, the sequence $(r_s)$ converges uniformly. Denote its limit by $r_\infty$. % Denote the limit of $(r_s)$ by $r_\infty$.
By the construction, the function $r_\infty$ is the pointwise limit of a decreasing sequence of the upper semicontinuous functions $r_s$, $s=1, 2, ...$, and hence $r_\infty$ is also upper semicontinuous.

\begin{rem}\label{cont-r-func}
If $X$ has the property that $\mathrm{loc.dim}(\cdot)$ is constant, then $$r_\infty(\tau) = \mathrm{rc}(A)(\tau(1_A)), \quad \tau \in \mathrm{T}^+(A). $$
\end{rem}

\begin{defn}\label{defn-gap-fctn}
Let $A$ be a simple C*-algebra. Define the set of (continuous) gap functions,  $G_A$, to be the set of continuous positive real-valued affine functions $h: \mathrm{T}^+(A) \to [0, +\infty)$, $0$ at $0$,  such that for any $a, b \in (A\otimes \mathcal K)^+$,
$$\mathrm{d}_\tau(a) + h(\tau) < \mathrm{d}_\tau(b),\ \tau \in \mathrm{T}^+(A)\quad  \Rightarrow \quad  a \precsim b. $$
\end{defn}

\begin{thm}\label{low-env-gn}
Let $A$ be a UHF-Villadsen algebra with seed space a (finite) simplicial complex. Then the upper semicontinuous positive real-valued affine function $r_\infty$ on $\tr^+(A)$ has the property
\begin{equation}\label{character-r}
\{h \in \mathrm{Aff}(\tr^+(A)): r_\infty \leq h \} = G_A .
\end{equation}
\end{thm}

\begin{rem}\label{uniq-crf}
Since $r_\infty$ is upper semicontinuous and affine (in particular concave), by Proposition I.1.2 of \cite{Alfsen-book},  one has
$$r_\infty = \inf\{h \in \mathrm{Aff}(\tr^+(A)): r_\infty \leq h\}.$$
Together with \eqref{character-r}, this implies
\begin{equation}\label{uniq-crf-eq}
r_\infty = \inf G_A.
\end{equation}
%and $G_A$ is characterized as the set of continuous positive real-valued affine functions which are larger than $r_\infty$. 
\end{rem}

\begin{proof}
To prove \eqref{character-r}, it is enough to show the following two properties:
\begin{enumerate}

\item If $h$ is a continuous positive real-valued affine function on $\mathrm{T}^+(A)$, $0$ at $0$, and $r_\infty \leq h$, then $h \in G_A$; that is, $h$ has the property that for any $a, b \in (A\otimes \mathcal K)^+$,
$$\mathrm{d}_\tau(a) + h(\tau) < \mathrm{d}_\tau(b),\ \tau \in \mathrm{T}^+(A) \quad  \Rightarrow \quad  a \precsim b. $$

\item If $h$ is a continuous positive real-valued affine function on $\mathrm{T}^+(A)$ and $h(\tau_0) < r_\infty(\tau_0)$ for some $\tau_0 \in \mathrm{T}^+(A)$, then $h \notin G_A$; that is, there are $a, b \in (A\otimes \mathcal K)^+$ such that $$\mathrm{d}_\tau(a) + h(\tau) < \mathrm{d}_\tau(b),\quad \tau \in  \mathrm{T}^+(A),$$ but $a$ is not Cuntz subequivalent to $b$.

%\item The function $r_\infty$ is upper semicontinuous, i.e., for any $\eps>0$ and any $\tau_0 \in \tr^+(A)$, there is an open neighbourhood $U$ of $\tau$ such that $$r_\infty(\tau) < r_\infty(\tau_0) + \eps,\quad \tau \in U.$$

\end{enumerate}

\noindent {\it Proof of (1).}
Let $h \in \mathrm{Aff}(\mathrm{T}^+(A))$ be continuous and $r_\infty \leq h$, and let $a, b \in (A\otimes \mathcal K)^+$ be such that
\begin{equation}\label{gap-given-cond-gn}
\mathrm{d}_\tau(a) + h(\tau) < \mathrm{d}_\tau(b),\quad \tau \in \mathrm{T}^+(A). 
\end{equation}

Let $\eps>0$ be arbitrary. There is $\delta > 0$ such that 
\begin{equation}\label{enlarged-gap-gn}
\mathrm{d}_\tau((a - \eps)_+) + h(\tau) + \delta < \mathrm{d}_\tau(b),\quad \tau \in \mathrm{T}^+(A). 
\end{equation}

Since $r_\infty \leq h$ and, by Lemma \ref{uniform-dec-r},  $(r_n)$ converges uniformly to $r_\infty$, there is $k$ such that
\begin{equation}\label{approx-r}
r_k(\tau) < h(\tau) + \frac{\delta}{8},\quad \tau \in \tr(A).
\end{equation} 
%With $s$ large enough, one has
%\begin{equation}\label{approx-r}
%r_k(\tau) < h(\tau) + \frac{\delta}{8},\quad \tau \in \tr(A_s).
%\end{equation}

Choose a non-zero trivial projection $q \in A_i$ for some $i \in \mathbb N$ such that
\begin{equation}\label{tr-q-gn}
\frac{3}{4}\delta < \tau(q) < \delta,\quad \tau \in A_i. 
\end{equation}

Since $A$ has stable rank one, by Theorem 8.11 of  \cite{Thiel-sr1} there is $c \in (A \otimes \mathcal K)^+$ such that $$\mathrm{d}_\tau(c) = h(\tau),\quad \tau \in \tr(A).$$ Therefore, by \eqref{gap-given-cond-gn} and \eqref{enlarged-gap-gn}, 
$$\mathrm{d}_\tau(a \oplus c \oplus q) < \mathrm{d}_\tau(b),\quad \tau \in \tr(A).$$
Note that, since $\tau \to \mathrm{d}_\tau(c) = h(\tau)$ is continuous, by Dini's theorem, there is $\delta'>0$ such that 
\begin{equation}\label{approx-h-1-gn}
h(\tau) -\frac{\delta}{4} = \mathrm{d}_\tau(c) - \frac{\delta}{4} < \tau(f_{\delta'}(c)) \leq \mathrm{d}_\tau(c) = h(\tau),\quad \tau \in\tr(A), 
\end{equation}
where $f_{\delta'}: \Real \to [0, 1]$ is the continuous function which is $0$ on $(-\infty, \delta']$, $1$ on $[2\delta', \infty)$, and linear in between. Fix $\delta'$.

Since $A$ is simple, by the proof of Proposition 3.2 of \cite{RorUHF-II}, there is $N \in \mathbb N$ such that
$$(a\oplus c \oplus q)\otimes 1_{N+1} \precsim  b \otimes 1_N.$$ 
By Lemma 5.6 of \cite{Niu-MD} (and its proof), for any $\eps'>0$ (to be determined later), there is $i \in \mathbb N$ such that there are positive elements $\tilde{a}$, $\tilde{c}$, and $\tilde{b}$ in $A_i$ (and $q \in A_i$) such that
$$ \norm{ a - \tilde{a} } < \eps',\quad \norm{c - \tilde{c}} < \eps', \quad \norm{b - \tilde{b}} < \eps', $$
$$ ((\tilde{a} - \eps')_+ \oplus (\tilde{c} - \eps')_+ \oplus q)\otimes 1_{N+1}  \precsim \tilde{b} \otimes 1_N, \quad \mathrm{and} \quad \tilde{b} \precsim b.$$

Hence
$$
\mathrm{d}_\tau(\tilde{a} - \eps')_+ + \mathrm{d}_\tau( \tilde{c} - \eps')_+ + \tau(q) < \mathrm{d}_\tau(\tilde{b}),\quad \tau \in \tr(A_i),
$$
and, by \eqref{tr-q-gn},
%In particular,
\begin{equation}\label{gap-finite-stage-gn}
\mathrm{d}_\tau(\tilde{a} - \eps')_+ + \mathrm{d}_\tau( \tilde{c} - \eps')_+ + \frac{3}{4}\delta < \mathrm{d}_\tau(\tilde{b}),\quad \tau \in \tr(A_i).
\end{equation}

Note that, with $\eps'$ sufficiently small, one has 
$$
\norm{ f_{\delta'}(c) - f_{\delta'}(\tilde{c})  } < \frac{\delta}{4}, %\quad \mathrm{and} \quad \norm{g_\eps(a) - g_{\eps}(\tilde{a})} < \frac{\delta}{8}
$$
and hence
\begin{equation}\label{approx-h-2-gn}
\abs{ \tau(f_{\delta'}(c)) - \tau(f_{\delta'}(\tilde{c}))  } < \frac{\delta}{4}, %\quad \mathrm{and} \quad \abs{ \tau(g_{\eps}(a)) - \tau(g_{\eps}(\tilde{a}))  } < \frac{\delta}{4}, 
\quad \tau \in \tr(A).
\end{equation}
Then, with $\eps' < \delta'$, by \eqref{approx-h-1-gn} and \eqref{approx-h-2-gn},
\begin{equation*}%\label{approx-h-3-gn}
\mathrm{d}_\tau(\tilde{c} - \eps')_+ \geq \tau(f_{\delta'}(\tilde{c})) >  \tau(f_{\delta'}(c)) - \frac{\delta}{4} > h(\tau) - \frac{\delta}{2}, \quad \tau \in \tr(A). 
\end{equation*}
By \eqref{approx-r} one has one more step, 
\begin{equation*}%\label{approx-h-3-gn}
\mathrm{d}_\tau(\tilde{c} - \eps')_+ %\geq \tau(f_{\delta'}(\tilde{c})) >  \tau(f_{\delta'}(c)) - \frac{\delta}{4} 
> h(\tau) - \frac{\delta}{2}> r_k(\tau) - \frac{3}{4}\delta, \quad \tau \in \tr(A). 
\end{equation*}
One should also assume that $\eps' < \eps$. Then, with $\eps'$ even smaller, % fix $\eps'$.
%Since $\tau \mapsto r_k(\tau)$ is upper semicontinuous, a compactness argument (Dini's theorem) shows that 
there is 
$\delta''>0$ such that
\begin{equation}\label{cut-ubd}
 \tau(f_{\delta''}((\tilde{c} - \eps')_+)) > r_k(\tau) - \frac{3}{4} \delta,\quad \tau \in \tr(A).  
 \end{equation}
Fix $\eps'$.

One asserts that there is $n$ such that 
\begin{equation}\label{cut-ubd-middle}
\tau(f_{\delta''}((\tilde{c} - \eps')_+)) > \tau(f_{k, n}) - \frac{3}{4} \delta  \geq r_k(\tau) - \frac{3}{4} \delta,\quad \tau \in \tr(A),
\end{equation}
where $(f_{k, n})_n$ is defined in \eqref{loc-sequence}, 
as, for any $\tau_0 \in \tr(A)$, by \eqref{cut-ubd}, there is $N \in \mathbb N$ such that
$$ \tau_0(f_{\delta''}((\tilde{c} - \eps')_+)) > \tau_0(f_{k, n}) - \frac{3}{4} \delta > r_k(\tau_0) - \frac{3}{4} \delta, \quad n >N.$$
Since $r_k$ is upper semicontinuous and $(f_{k, n})_n$ is decreasing, there is a neighbourhood $U \subseteq \tr(A)$ of $\tau_0$ such that
$$ \tau(f_{\delta''}((\tilde{c} - \eps')_+)) > \tau(f_{k, n}) - \frac{3}{4} \delta > r_k(\tau) - \frac{3}{4} \delta, \quad n >N,\quad \tau \in U.$$
Then, a compactness argument shows \eqref{cut-ubd-middle}.

Then, by \eqref{cut-ubd-middle}, there is $s > \max\{ i, k\}$ such that 
\begin{equation*}%\label{cut-ubd-middle-finite}
 \tau(f_{\delta''}((\tilde{c} - \eps')_+)) >\tau(f_{k, n}) - \frac{3}{4} \delta> r_k(\tau) - \frac{3}{4} \delta,\quad \tau \in \tr(A_s). 
 \end{equation*}
As, otherwise, there is a sequence $\tau_{s_1} \in \tr(A_{s_1}), \tau_{s_2} \in \tr(A_{s_2}), ...$ such that 
$$ \tau_{s_i}(f_{\delta''}((\tilde{c} - \eps')_+)) \leq \tau_{s_i}(f_{k, n}) - \frac{3}{4} \delta,\quad \tau \in \tr(A_s).  $$
Extend each $\tau_{s_i}$, $i=1, 2, ...$, to a state of $A$, and pick $\tau_\infty$ to be an accumulation point of $\{\tau_{s_i},\ i=1, 2, ... \}$. Then $\tau_\infty \in \tr(A)$ and it fails to satisfy \eqref{cut-ubd-middle}.

Therefore,
\begin{equation}\label{lbd-c-local-gn}
 \mathrm{d}_\tau((\tilde{c} - \eps')_+) \geq \tau(f_{\delta''}((\tilde{c} - \eps')_+)) > r_k(\tau) - \frac{3}{4} \delta,\quad \tau \in \tr(A_s).
 \end{equation}
%
%\vskip 1in 
%
%
%Since $\tau \mapsto h(\tau)$ is continuous, by \eqref{approx-h-3-gn} and Dini's Theorem, there is $\delta''>0$ such that
%$$ f_{\delta''}((\tilde{c} - \eps')_+) > h(\tau) - \frac{5}{8} \delta,\quad \tau \in \tr(A).  $$
%Thus, there is $s > i$ sufficiently large such that
%$$ f_{\delta''}((\tilde{c} - \eps')_+) > h(\tau) - \frac{5}{8} \delta,\quad \tau \in \tr(A_s),$$
%where $c_i$ and $h$ are regarded as elements of $A_s$, and therefore 
%\begin{equation}\label{lbd-c-local-gn}
% \mathrm{d}_\tau((\tilde{c} - \eps')_+) \geq f_{\delta''}((\tilde{c} - \eps')_+) > h(\tau) - \frac{5}{8} \delta,\quad \tau \in \tr(A_s).
% \end{equation}
% 
Then, by \eqref{lbd-c-local-gn} and \eqref{gap-finite-stage-gn}, for all $\tau \in \tr(A_s)$, one has  
\begin{eqnarray*}
\mathrm{d}_\tau((\tilde{a} - \eps)_+) + r_{s}(\tau)  \leq \mathrm{d}_\tau((\tilde{a} - \eps)_+) + r_k(\tau)   <   \mathrm{d}_\tau((\tilde{a} - \eps')_+) +  \mathrm{d}_\tau((\tilde{c} - \eps')_+) + \frac{3}{4} \delta 
 <  \mathrm{d}_\tau(\tilde{b}).
\end{eqnarray*}
%and, hence, 
%\begin{equation}
%\mathrm{d}_\tau((\tilde{a} - \eps)_+) + r_k(x) < \mathrm{d}_\tau((\tilde{a} - \eps)_+) + h(\tau) + \frac{\delta}{8} < \mathrm{d}_\tau(\tilde{b}),\quad \tau \in \tr(A_s).
%\end{equation}
%
In particular, 
\begin{equation}\label{RSD-gap-gn}
 \mathrm{rank}((\tilde{a} - \eps)_+(x)) + \frac{1}{2} \cdot \mathrm{loc.dim}(x) <  \mathrm{rank}(\tilde{b}(x)),\quad x \in X^{n_1 \cdots n_{s-1}}.
\end{equation}

Since $X$ is a (finite) simplicial complex, writing $$\mathrm{loc.dim}(X) = \{d_1, d_2, ..., d_l\},$$ where $d_1 > d_2 > \cdots > d_l$, and defining
$$X_{d_i} = \{x \in X^{n_1\cdots n_{s-1}}: \mathrm{loc.dim}(x) = d_i\},$$
one has $$\mathrm{dim}(\overline{X_{d_i}}) = d_i,\quad i=1, ..., l.$$

Note that there is a decomposition $$X^{n_1 \cdots n_{s-1}} = X_{d_1} \cup X_{d_2} \cup \cdots \cup X_{d_l},$$
and, since the function $\mathrm{loc.dim}(\cdot)$ is upper semicontinuous, 
the sets $$ Y_i:=X_{d_1} \cup \cdots \cup X_{d_i},\quad i=1, ..., l, $$
are closed.  This induces a recursive subhomogeneous decomposition (see \cite{Phill-RSA1}) $$A_s = \mathrm{M}_{(n_1 + k_1) \cdots (n_{s-1} + k_{s-1})}(\mathrm{C}(X^{n_1 \cdots n_{s-1}})) = (\cdots ((A_1 \oplus_{A_2^{(0)}} A_2) \oplus_{A_3^{(0)}} A_3) \oplus \cdots )\oplus_{A^{(0)}_{l}}  A_{l},$$
where $$A_i = \mathrm{M}_{(n_1+ k_1) \cdots (n_{s-1}+k_{s-1})}(\mathrm{C}(\overline{X_{d_i}} )) \quad \mathrm{and} \quad A_i^{(0)} = \mathrm{M}_{(n_1+ k_1) \cdots (n_{s-1} + k_{s-1})}(\mathrm{C}(Y_{i-1} \cap \overline{X_{d_i} })).$$

For each $X_i$, $i=1, ..., l$, one has 
$$\mathrm{dim}(\overline{X_i}) = d_i = \mathrm{loc.dim}(x),\quad x \in X_i.$$ Thus, by \eqref{RSD-gap-gn},
$$\mathrm{rank}((\tilde{a} - \eps)_+(x)) + \frac{1}{2} \cdot \mathrm{dim}(\overline{X_i}) <  \mathrm{rank}(\tilde{b}(x)),\quad x \in X_{i},\ i=1, ..., l.
$$
By Theorem 4.6 of \cite{Toms-Comp-DS},
$$(\tilde{a} - \eps)_+ \precsim \tilde{b} \precsim b.$$
Since $$a \approx_{\eps'} \tilde{a} \approx_\eps (\tilde{a} - \eps)_+,$$
one has 
$$(a - 2\eps)_+ \precsim b.$$
Since $\eps$ is arbitrary, this implies $a \precsim b$. This shows (1).
%\end{proof}

\noindent {\it Proof of (2).}
Let $h \in \mathrm{Aff}^+(\mathrm{T}^+(A))$ such that $$h(\tau_0) < r_\infty(\tau_0)$$ for some $\tau_0 \in \mathrm{T}^+(A)$. Let us show that $h \notin G_A$.

Set 
$$\delta = \max\{r_\infty(\tau) - h(\tau): \tau \in \mathrm{T}^+(A),\ \tau(1_A) = 1\} > 0$$
and 
$$ M = \max\{h(\tau): \tau \in \mathrm{T}^+(A),\ \tau(1_A) = 1 \}. $$

By \eqref{Vill-lim}, one has the following inductive limit decomposition of the ordered Banach space $\mathrm{Aff}(\mathrm{T}(A))$:
$$
\xymatrix{
\mathrm{C}_{\Real}(X) \ar[r]^-{\varphi_1^*} & \mathrm{C}_{\Real}(X^{n_1}) \ar[r]^{\varphi_2^*} &  \mathrm{C}_{\Real}(X^{n_1n_2}) \ar[r]^-{\varphi_3^*} & \cdots \ar[r] & \mathrm{Aff}(\mathrm{T}(A)).
}
$$
Then there is $h_s \in \mathrm{C}_{\Real^+}(X^{n_1 \cdots n_s})$ such that
\begin{equation}\label{approx-h-gn}
 \norm{ h _s - h}_\infty < \frac{\delta}{4}. 
 \end{equation}
Hence, with $s$ sufficiently large, 
there is $\tau_0 \in \tr(A_s) = \mathcal M_1(X^{n_1\cdots n_s})$ such that
$$ h_s(\tau_0) <  r_s(\tau_0) - \frac{3}{4}\delta,$$
and this implies that there is $x_0 \in X^{n_1 \cdots n_s}$ such that
$$h_s(x_0) < r_s(x_0) - \frac{3}{4} \delta = \frac{1}{2} \cdot \frac{\mathrm{loc.dim}(x_0)}{(n_1+k_1) \cdots (n_s + k_s)} - \frac{3}{4} \delta.$$

Since $X^{n_1\cdots n_s}$ is a (finite) simplical complex, there is a $\mathrm{loc.dim}(x_0)$-dimensional ball $B_s \subseteq X^{n_1 \cdots n_s}$ in any neighbourhood of $x_0$. Then, since $h_s$ is continuous, there is a Euclidean ball $B_s \subseteq X^{n_1 \cdots n_s}$ (in any neighbourhood of $x_0$) with dimension $d_{x_0}$, where
$$
d_{x_0} = \left\{ 
\begin{array}{ll}
\mathrm{loc.dim}(x_0), & \textrm{if $\mathrm{loc.dim}(x_0)$ is odd}, \\
\mathrm{loc.dim}(x_0) - 1, & \textrm{if $\mathrm{loc.dim}(x_0)$ is even},
\end{array}
\right.
$$
 %dimension  $\mathrm{loc.dim}(x_0)$ or  $(\mathrm{loc.dim}(x_0)-1)$ 
 such that 
\begin{equation}\label{local-small-1-gn}
h_s(x) < r_s(x) - \frac{3}{4} \delta =  r_s(x_0) - \frac{3}{4} \delta ,\quad x \in B_s.
\end{equation}
%(In the case that $\mathrm{loc.dim}(x_0)$ is even, choose a $(\mathrm{loc.dim}(x_0)-1)$-dimensional ball $B_s$.)

One should also assume $s$ is sufficiently large that
\begin{equation}\label{large-s-2-gn}
\abs{\frac{\mathrm{loc.dim}(x) - 1}{(n_1 + k_1) \cdots (n_s + k_s)} - r_s(x)} < \frac{\delta}{4}
\end{equation}
and
\begin{equation}\label{large-s-3-gn}
1 - (\frac{n_{s+1}}{n_{s+1} + k_{s+1}}) \cdots (\frac{n_t}{n_t + k_t}) <  \frac{\delta}{8},\quad t > s.
\end{equation}

Over $\partial B_s$, which is a $(d_{x_0} - 1)$-dimensional sphere, there is a complex vector bundle $E$ such that
$$\mathrm{rank}(E) = \frac{1}{2}(d_{x_0} - 1) \quad \mathrm{and} \quad c_{\frac{d_{x_0} - 1}{2}} \in H^{d_{x_0} - 1}(S^{d_{x_0} - 1}) \setminus \{0\}.$$
(Such a vector bundle exists, as, otherwise, the $\frac{d_{x_0} - 1}{2}$-th Chern class of every vector bundle would be trivial, and 
then the Chern character would not induce a rational isomorphism between the K-group and the
cohomology group of the sphere $S^{d_{x_0} - 1}$.)

Denote by $p$ the projection associated to $E_i$. Then, 
\begin{equation}\label{rank-p-ball-gn}
\mathrm{tr}_x(p) = \frac{1}{2} \cdot \frac{d_{x_0} - 1}{(n_1 + k_1) \cdots (n_s + k_s)}, \quad x \in \partial B_s,
\end{equation} 
where $\mathrm{tr}_x$ is the tracial state of $A_s$ which is concentrated at $x$.
Extend $p$ to a positive element of $A_s \otimes \mathcal K = \mathrm{C}(X^{n_1 \cdots n_s}) \otimes \mathcal K $ and still denote it by $p$.

Choose a positive matrix $e$ such that $$\mathrm{rank}(e) \geq (M+2\delta)(n_1 + k_1) \cdots (n_s + k_s),$$ and set $$p_0: X^{n_1\cdots n_s} \ni x \mapsto \mathrm{dist}(x, \partial B_s) e \in \mathcal K.$$ Consider the element $p + p_0$, and still denote it by $p$. Then, together with \eqref{rank-p-ball-gn}, \eqref{large-s-2-gn}, and \eqref{local-small-1-gn},   one has
\begin{equation}\label{lbd-p-gn}
\tau(p) > h_s(\tau) + \frac{\delta}{2},\quad \tau \in \tr(A_s), 
\end{equation}
and the restriction of $p$ to $\partial B_s$ is a projection such that the corresponding vector bundle has non-zero total Chern class at degree $\mathrm{loc.dim}(x_0) - 1$.

Let $q \in A_s$ be a trivial projection with $$\frac{\delta}{4} > \tau(q) > \frac{\delta}{8},\quad \tau \in \tr(A_s).$$ Then, by \eqref{approx-h-gn} and \eqref{lbd-p-gn},
$$ \tau(q) + h(\tau) < \frac{\delta}{4} + (h_s(\tau) + \frac{\delta}{4}) < \tau(p),\quad \tau \in \tr(A).$$
To show the theorem, it is enough to show that $q$ is not Cuntz sub-equivalent to $p$.

Let $t > s$ be arbitrary, and consider the building block $A_t$.  Consider the closed subset
$$ C_t:=\underbrace{\partial B_s \times \cdots \times \partial B_s}_{n_{s+1} \cdots n_{t}} \subseteq \underbrace{X^{n_1\cdots n_s} \times \cdots \times X^{n_1\cdots n_s}}_{n_{s+1} \cdots n_{t}}.$$
Then
$$\phi_{s, t}(p)|_{C_t} = \mathrm{diag}\{p|_{\partial B_s}\circ \pi_1, ..., p|_{\partial B_s} \circ \pi_{n_{s+1} \cdots n_t}, c\},$$
where $c$ is a constant positive matrix of rank at most $$(n_1 + k_1)\cdots(n_t + k_t) - (n_1 + k_1)\cdots(n_s + k_s)(n_{s+1} \cdots n_t).$$
Hence the positive element $ \phi_{s, t}(p)|_{C_t} $ is Cuntz equivalent to a projection of rank at most 
$$ R_t:=\frac{1}{2}(d_{x_0} - 1)(n_{s+1} \cdots n_t) + (n_1 + k_1)\cdots(n_t + k_t) - (n_1 + k_1)\cdots(n_s + k_s)(n_{s+1} \cdots n_t)$$
and with non-zero total Chern class (by the K{\"{u}}nneth Theorem) at 
$$H^{(d_{x_0} - 1)(n_{s+1} \cdots n_t)}(C_t).$$
Thus, by Remark 3.2 of \cite{ELN-Vill}, the trivial subprojection of $[ \phi_{s, t}(p)|_{C_t} ]$ has rank at most
$$ R_t - \frac{1}{2}(d_{x_0} - 1)(n_{s+1} \cdots n_t),$$
and hence has (normalized) trace at most
\begin{eqnarray*}
&&\frac{R_t - \frac{1}{2}(d_{x_0} - 1)(n_{s+1} \cdots n_t)}{(n_1 + k_1) \cdots (n_t + k_t)} \\
& = & \frac{ (n_1 + k_1)\cdots(n_t + k_t) - (n_1 + k_1)\cdots(n_s + k_s)(n_{s+1} \cdots n_t) }{(n_1 + k_1) \cdots (n_t + k_t)} \\
& = & 1 - \frac{n_{s+1}}{n_{s+1} + k_{s+1}} \cdots \frac{n_t}{n_t + k_t} \\
& < & \frac{\delta}{8}.
\end{eqnarray*}
Since $\tau(\phi_{s, t}(q)) > \delta/8$ for all $\tau \in \tr(A_t)$, this implies that $q$ is not Cuntz subequivalent to $p$. This shows (2).
\end{proof}

\begin{cor}\label{different-seed}
Let $X_1 = [0, 1]^2$ and $X_2 = [0, 1] \vee [0, 1]^2$, and let $A_1 = A(X_1, (n_i), (k_i))$ and $A_2 = (X_2, (n_i), (k_i))$ be UHF-Villadsen algebras with seed spaces $X_1$ and $X_2$ respectively. Then the function $r_\infty(A_1)$ is constant on $\tr(A_1)$, while the function $r_\infty(A_2)$ is not constant on $\tr(A_2)$.

In particular, $$\mathrm{Cu}(A_1) \ncong \mathrm{Cu}(A_2),$$ and $$\mathrm{rc}(A_1) = \mathrm{rc}(A_2) \quad \textrm{but} \quad A_1 \otimes \mathcal K  \ncong A_2 \otimes \mathcal K.$$ 
\end{cor}
\begin{proof}
By Remark \ref{cont-r-func}, the affine function $r_\infty{(A_1)}$ is constant on $\tr(A_1)$, and so it factors through $\tr^+(A_1) \to \tr^+(\Kzero(A_1)) \cong \Real^+$. On the other hand, as we will show below, the function $r_\infty{(A_2)}$ is not constant on $\tr(A_2)$, and so does not factor through $\tr^+(A_2) \to \tr^+(\Kzero(A_2)) \cong \Real^+$. By Theorem \ref{low-env-gn}, the functions $r_\infty{(A_1)}$ and $r_\infty{(A_2)}$ are invariant (uniquely determined by the Cuntz semigroup). Hence, $\mathrm{Cu}(A_1) \ncong \mathrm{Cu}(A_2)$, as desired.

Let us show that the function $r_\infty{(A_2)}$ is not constant on $\tr(A_2)$. Let $x \in [0, 1] \vee [0, 1]^2 $ be a point of $[0, 1] \setminus\{*\}$, and let $y \in [0, 1] \vee [0, 1]^2 $ be a point of $[0, 1]^2\setminus\{*\}$, where $*$ is the common point of $[0, 1]$ and $[0, 1]^2$.  Note that
\begin{equation}\label{local-dim-comp}
\mathrm{loc.dim}((x, ..., x)) = n_1\cdots n_{s} \quad \mathrm{and} \quad \mathrm{loc.dim}((y, ..., y)) = 2 n_1\cdots n_{s} 
\end{equation} 
on $X_2^{n_1\cdots n_s}$. 

Put
$$\frac{1}{2}(\frac{n_1}{n_1 + k_1})  (\frac{n_2}{n_2 + k_2})   \cdots = \gamma,$$
and note that $\gamma \neq 0$.

On each $s=1, 2, ...$, consider the the tracial states $\tau_{x, s},  \tau_{y, s} \in \tr(\mathrm{M}_{(n_1+k_1)\cdots (n_s + k_s)}(\mathrm{C}(X_2^{n_1\cdots n_s})))$ which are induced by the Dirac measures which are concentrated on $$(x, ..., x) \quad \mathrm{and} \quad (y, ..., y),$$ respectively. Extend $\tau_{x, s},  \tau_{y, s}$ to states of $A_2$ and still denote them by $\tau_{x, s},  \tau_{y, s}$.

Note that the sequences $(\tau_{x, s})_s$ and $(\tau_{y, s})_s$ converge pointwise to $\tau_x$ and $\tau_y$, which are tracial states of $A_2$. Let us show that
$$r_\infty(\tau_x) = \gamma \quad \mathrm{and} \quad r_\infty(\tau_y) = 2\gamma.$$
In particular, $$r_\infty(\tau_x) \neq r_\infty(\tau_y).$$

Let $\eps>0$ be arbitrary.
By Lemma \ref{uniform-dec-r} and Equation \eqref{rdg-cond}, there is $s_0$ such that
\begin{equation}\label{far-eq-1}
 \norm{r_\infty - r_{s_0}}_\infty < \eps,
 \end{equation}
\begin{equation}\label{far-eq-2}
 1 - (\frac{n_{s_0}}{n_{s_0} + k_{s_0}}) \cdots (\frac{n_s}{n_s + k_s}) < \eps, \quad s \geq s_0,
 \end{equation}
 and
 \begin{equation} \label{far-eq-3}
 \gamma \approx_{\eps} \frac{1}{2} \cdot \frac{n_1 \cdots n_s}{(n_1 +k_1) \cdots (n_s + k_s)},\quad s \geq s_0.
 \end{equation}

By the choice of $x$ and $y$, there are positive real-valued continuous functions $f_{s_0}, g_{s_0} \in \mathrm{C}(X_2^{n_1 \cdots n_{s_0-1}}) $ such that 
\begin{equation}\label{sandwich}
 g_{s_0} \leq r_{s_0} \leq  f_{s_0},
 \end{equation}
$$ g_{s_0}((x, ..., x)) = r_{s_0}((x, ..., x)), \quad   g_{s_0}((y, ..., y)) = r_{s_0}((y, ..., y)),$$
and
$$ f_{s_0}((x, ..., x)) = r_{s_0}((x, ..., x)), \quad   f_{s_0}((y, ..., y)) = r_{s_0}((y, ..., y)).$$
By \eqref{local-dim-comp}, \eqref{far-eq-2}, and \eqref{far-eq-3}, for all $s > s_0$, one has
\begin{eqnarray*} 
&& \tau_{x, s}(\phi_{s_0, s}(f_{s_0})) \\
& = & \frac{n_{s_0} \cdots n_s}{(n_{s_0 } + k_{s_0 }) \cdots (n_{s} + k_{s})} f_{s_0}((x, ..., x)) + (1 - \frac{n_{s_0} \cdots n_s}{(n_{s_0 } + k_{s_0 }) \cdots (n_{s} + k_{s})})(f_{s_0}(\cdot) + \cdots f_{s_0}(\cdot)) \\
& \approx_{\eps} &  \frac{n_{s_0} \cdots n_s}{(n_{s_0 } + k_{s_0 }) \cdots (n_{s} + k_{s})} f_{s_0}((x, ..., x)) \\
& = &  \frac{n_{s_0} \cdots n_s}{(n_{s_0 } + k_{s_0 }) \cdots (n_{s} + k_{s})} r_{s_0}((x, ..., x)) \\
& = & \frac{1}{2} \cdot  \frac{\mathrm{loc.dim}((x, ..., x))}{(n_1 + k_1) \cdots (n_{s-1} + k_{s-1})} \frac{n_{s_0} \cdots n_s}{(n_{s_0 } + k_{s_0 }) \cdots (n_{s} + k_{s})} \\
& = & \frac{1}{2} \cdot  \frac{n_1 \cdots n_{s_0 - 1}}{(n_1 + k_1) \cdots (n_{s-1} + k_{s-1})} \frac{n_{s_0} \cdots n_s}{(n_{s_0 } + k_{s_0 }) \cdots (n_{s} + k_{s})} \\
& \approx_{\eps} & \gamma,
\end{eqnarray*}
and 
\begin{eqnarray*} 
&  & \tau_{x, s}(\phi_{s_0, s}(g_{s_0})) \\
%& = & \frac{n_{s_0} \cdots n_s}{(n_{s_0 } + k_{s_0 }) \cdots (n_{s} + k_{s})} g_{s_0}((y, ..., y)) + (1 - \frac{n_{s_0} \cdots n_s}{(n_{s_0 } + k_{s_0 }) \cdots (n_{s} + k_{s})}) \\
& \approx_{\eps} &  \frac{n_{s_0} \cdots n_s}{(n_{s_0 } + k_{s_0 }) \cdots (n_{s} + k_{s})} g_{s_0}((x, ..., x)) \\
& = &  \frac{n_{s_0} \cdots n_s}{(n_{s_0 } + k_{s_0 }) \cdots (n_{s} + k_{s})} r_{s_0}((x, ..., x)) \\
& = & \frac{1}{2} \cdot  \frac{\mathrm{loc.dim}((x, ..., x))}{(n_1 + k_1) \cdots (n_{s-1} + k_{s-1})} \frac{n_{s_0} \cdots n_s}{(n_{s_0 } + k_{s_0 }) \cdots (n_{s} + k_{s})} \\
& = & \frac{1}{2} \cdot  \frac{ n_1 \cdots n_{s_0 - 1}}{(n_1 + k_1) \cdots (n_{s-1} + k_{s-1})} \frac{n_{s_0} \cdots n_s}{(n_{s_0 } + k_{s_0 }) \cdots (n_{s} + k_{s})} \\
& \approx_{\eps} &  \gamma.
\end{eqnarray*}
Similarly, for all $s > s_0$, one has
$$ \abs{\tau_{y, s}(\phi_{s_0, s}(f_{s_0})) - 2\gamma} < 3\eps \quad \mathrm{and} \quad \abs{\tau_{y, s}(\phi_{s_0, s}(g_{s_0})) - 2\gamma} < 3 \eps. $$ 
Thus, by \eqref{sandwich} and \eqref{far-eq-1}, one has
$$r_\infty(\tau_x) = \lim_{s \to \infty} r_s(\tau_x) < r_{s_0}(\tau_x) < \widehat{f_{s_0}}(\tau_x) = \lim_{s \to\infty} \tau_{x, s}(\phi_{s_0, s}(f_{s_0})) < \gamma   + 2 \eps,$$
$$ r_\infty(\tau_x) > r_{s_0}(\tau_x) - \eps > \widehat{g_{s_0}}(\tau_x) - \eps = \lim_{s \to \infty} \tau_{x, s}(\phi_{s_0, s}(g_{s_0})) - \eps > \gamma - 3 \eps,$$
and hence
$$\gamma - 3 \eps < r_\infty(\tau_x) < \gamma + 2\eps. $$
The same argument applied to $y$ shows that 
$$2\gamma - 4 \eps < r_\infty(\tau_y) < 2\gamma + 3\eps. $$
Since $\eps$ is arbitrary, one has
$$r_\infty(\tau_x) = \gamma \quad\mathrm{and} \quad r_\infty(\tau_y) = 2\gamma,$$
as asserted.
\end{proof}

\begin{cor}\label{joint-classification}
Let $A = A(X^{(A)}, (n^{(A)}_i), (k^{(A)}_i))$ and $B = B(X^{(B)}, (n^{(B)}_i), (k^{(B)}_i))$  be UHF-Villadsen algebras with seed spaces $[0, 1]^2$ or $ [0, 1] \vee [0, 1]^2$. Then $A \cong B$ if, and only if, $(\mathrm{Cu}(A), [1_A]) \cong (\mathrm{Cu}(B), [1_B])$. Indeed, $A \cong B$ if, and only if, $$((\Kzero(A), \Kzero^+(A), [1_A]_0), r_\infty{(A)}) \cong ((\Kzero(B), \Kzero^+(B), [1_B]_0), r_\infty(B)).$$
\end{cor}
\begin{proof}
Since $\mathrm{Cu}(A) \cong \mathrm{Cu}(B)$, one has $r_\infty(A) = r_\infty(B)$ under the isomorphism. Therefore the restrictions of these functions to the trace simplexes, respectively, are both constant or not, and hence, by Corollary \ref{different-seed}, $X^{(A)} \cong X^{(B)}$. Note that $\mathrm{rc}(A) = \mathrm{rc}(B)$ (see Theorem \ref{rc-V-general} below). The corollary then follows from Theorem 7.1 of \cite{ELN-Vill}.
\end{proof}

\begin{rem}
Note that, since $X_1$ and $X_2$ are contractible, by Corollary 6.1 of  \cite{West-cube},  $$X_1^\infty \cong X_2^\infty.$$ So, the Cuntz semigroup of a Villadsen algebra contains information which is finer than the infinite Cartesian power of the seed space.

On the other hand, the Villadsen algebras $A([0, 1], (n_i), (k_i))$ and $A([0, 1]^2, (n_i), (k_i))$ are stably isomorphic (but not isomorphic). Therefore, their Cuntz semigroups are isomorphic.  
\end{rem}

\begin{rem}
Let $X_1 = [0, 1] \vee [0, 1]^2$ and $X_2 = [0, 1] \vee [0, 1]^2 \vee [0, 1]$. It would be interesting to know whether the Villadsen algebras $A_1 = A(X_1, (n_i), (k_i)) $ and $A_2 = A(X_2, (n_i), (k_i))$ share the same Cuntz semigroup.
\end{rem}

%\begin{lem}\label{uniq-crf}
%Let $A$ be a unital C*-algebra. There is at most one upper semicontinuous positive real valued affine function on $\tr^+(A)$ such that 
%$$ G_A = \{h \in \mathrm{Aff}(\tr^+(A)):  r_\infty \leq h \}.$$
%%
%%which has Properties (1), (2), and (3) of Theorem \ref{low-env-gn}.
%\end{lem}
%\begin{proof}
%Let $\tilde{r}_\infty$ be another upper semicontinuous positive real valued affine function which satisfies the property above. %has Properties (1), (2), and  (3) of Theorem \ref{low-env-gn}. 
%Assume there were $\tau_0 \in \tr^+(A)$ such that $\tilde{r}_\infty(\tau_0) < r_\infty(\tau_0)$. Since $\tilde{r}_\infty$ is affine and upper semicontinuous, %there is a sequence $(h_n) \subseteq \mathrm{Aff}(\tr(A))$ which converges pointwisely to $\tilde{r}_\infty$ from above. Then, with $n$ sufficiently large, 
%one has $$\tilde{r}_\infty = \inf\{h \in \mathrm{Aff}(\tr^+(A)): h > \tilde{r}_\infty\},$$
%and hence there is $h \in \mathrm{Aff}(\tr(A))$ such that $$ \tilde{r}_\infty(\tau_0) \leq h(\tau_0) < r_\infty(\tau_0).$$ Since $\tilde{r}_\infty$ has Property (1), one has $h \in G_A$. But since $r_\infty$ has Property (2), $h \notin G_A$, which is a contradiction.
%
%The same argument also shows that there is no $\tau \in \tr(A)$ such that ${r}_\infty(\tau) < \tilde{r}_\infty(\tau))$. Therefore $$\tilde{r}_\infty(\tau) = r_\infty(\tau),\quad \tau \in \tr(A),$$
%as asserted.
%\end{proof}

\begin{cor}[cf.~Corollary \ref{aut-fix}]\label{aut-tail}
Let $A(X, (n_i), (k_i))$ be a UHF-Villadsen algebra with $X$ a finite simplicial complex, and let $\sigma \in \mathrm{Aut}(A)$. Then $$r_\infty((\sigma^*(\tau))) = r_\infty(\tau),\quad \tau \in \tr(A).$$
\end{cor}

\begin{proof}
Since $\sigma$ is an automorphism, one has $\sigma_*(G_A) = G_A$, and therefore (by \eqref{uniq-crf-eq}) $\sigma_*(r_\infty)$ is also the infimum of $G_A$. Hence 
$\sigma_*(r_\infty) = r_\infty$, as asserted.
\end{proof}

\begin{cor}[cf.~Example \ref{exm-HP}]\label{transitivity}
Let $A = A(X, (n_i), (k_i))$ be a UHF-Villadsen algebra with seed space $X = [0, 1] \vee [0, 1]^2$. Then the action of $\mathrm{Aut}(A)$ on the extreme points of $\tr(A)$, the Poulsen simplex (see \cite{ELN-Vill}), is not transitive.
\end{cor}
\begin{proof}
The restriction of the function $r_\infty$ to $\tr(A)$ is not constant, and so there are $\tau_1, \tau_2 \in \partial \tr(A)$ such that $r_\infty(\tau_1) \neq r_\infty(\tau_2)$. By the corollary above, there is no $\sigma \in \mathrm{Aut}(A)$ such that $\sigma^*(\tau_1) = \tau_2$, as desired.
\end{proof}

\begin{defn}
Let $A$ be a C*-algebra. An upper semicontinuous extended positive real valued affine function $r_\infty$ on $\mathrm{T}^+(A)$ will be called the comparison radius function if it has the following property:
$$ \{h \in \mathrm{Aff}(\tr^+(A)):  r_\infty \leq h \} = G_A. $$

%\begin{enumerate}
%
%\item If $h \in \mathrm{Aff}(\mathrm{T}^+(A))$ is continuous and $r_\infty \leq h$, then $h \in G_A$. %that is, $h$ has the property that for any $a, b \in (A\otimes \mathcal K)^+$,
%%$$\mathrm{d}_\tau(a) + h(\tau) < \mathrm{d}_\tau(b),\ \tau \in \mathrm{T}^+(A) \quad  \Rightarrow \quad  a \precsim b. $$
%
%\item If $h \in \mathrm{Aff}(\mathrm{T}^+(A))$ is continuous and $h(\tau_0) < r_\infty(\tau_0)$ for some $\tau_0 \in \mathrm{T}^+(A)$, then, $h \notin G_A$. % that is, there are $a, b \in (A\otimes \mathcal K)^+$ such that $$\mathrm{d}_\tau(a) + h(\tau) < \mathrm{d}_\tau(b),\quad \tau \in  \mathrm{T}^+(A),$$ but $a$ is not subequivalent to $b$.
%
%\item The function $r_\infty$ is a decreasing pointwise limit of $\mathrm{Aff}(\mathrm{T}(A))$ ($r$ is upper semicontinuous).
%
%\end{enumerate}
%In other words, the function $r_\infty$ has Properties (1), (2), and (3) of Theorem \ref{low-env-gn}.

Note that, by Remark \ref{uniq-crf}, the comparison radius function, if it exists, satisfies $$r_\infty = \inf G_A,$$ and hence is unique.
\end{defn}

The radius of comparison can be recovered from the comparison radius function $r_\infty$ (cf.~Remark \ref{cont-r-func}). % (trivial for UHF Villadsen algebras).
\begin{thm}\label{rc-V-general}
Let $A$ be a C*-algebra such that $\tr(A) \neq \O$ and the comparison radius function $r_\infty$ exists (e.g., $A$ is a UHF-Villadsen algebra with seed space a (finite) simplicial complex). %let $r$ be an affine function which has Properties (1), (2), (3) of Theorem \ref{low-env-gn}. 
Then, for any non-zero projection $p \in A \otimes \mathcal K$, one has
$$\mathrm{rc}(p(A \otimes \mathcal K)p) = \sup\{r_\infty(\tau): \tau(p) = 1,\ \tau \in \mathrm{T}^+(A)\}.$$
\end{thm}
\begin{proof}
Let $s > \sup\{r_\infty(\tau): \tau(p) = 1,\ \tau \in \mathrm{T}^+(A)\}$ be a real number. Then, regarding $s$ as a constant (continuous) affine function on the section $\{\tau \in \mathrm{T}^+(A): \tau(p)=1\} = \mathrm{T}(p(A\otimes\mathcal K)p)$, and extending $s$ to $\mathrm{T}^+(p(A \otimes\mathcal K)p) = \mathrm{T}^+(A)$, one has $$r_\infty(\tau) < s(\tau), \quad \tau \in \mathrm{T}^+(A).$$ Therefore $s \in G_A$ (see (1) of Theorem \ref{low-env-gn}), and so $s$ has the property
$$\mathrm{d}_\tau(a) + s < \mathrm{d}_\tau(b),\ \tau \in \mathrm{T}(p(A\otimes\mathcal K)p) \quad \Longrightarrow \quad a \precsim b,\quad a, b \in (A \otimes \mathcal K)^+,$$ and hence $\mathrm{rc}(p(A\otimes \mathcal K)p) \leq s. $ This shows that $$\mathrm{rc}(p(A \otimes \mathcal K)p) \leq \sup\{r_\infty(\tau): \tau(p) = 1,\ \tau \in \mathrm{T}^+(A)\}.$$

Now, let $s \leq \sup\{r_\infty(\tau): \tau(p) = 1,\ \tau \in \mathrm{T}^+(A)\}$ be a real number. Then, for an arbitrary $\eps>0$, there is $\tau_0 \in \{\tau \in \mathrm{T}^+(A): \tau(p)=1\}$ such that $s - \eps < r(\tau_0)$. Regarding $s - \eps$ as a continuous affine function on $\mathrm{T}^+(A)$ as above (constant equal to the number $s - \eps$ on $\mathrm{T}(p(A \otimes \mathcal K)p)$), one has $s - \eps \notin G_A$ (by (2) of Theorem \ref{low-env-gn}); that is, there are $a, b \in (A \otimes \mathcal K)^+$ such that $$ \mathrm{d}_\tau(a) + (s - \eps) < \mathrm{d}_\tau(b),\ \tau \in \mathrm{T}(p(A\otimes\mathcal K)p),$$ but $a$ is not Cuntz subequivalent to $b$. Therefore, $$s - \eps \leq \mathrm{rc}(p(A\otimes \mathcal K)p).$$ Since $\eps$ is arbitrary, this implies $s \leq \mathrm{rc}(p(A\otimes \mathcal K)p)$. This shows that $$ \sup\{r_\infty(\tau): \tau(p) = 1,\ \tau \in \mathrm{T}^+(A)\} \leq \mathrm{rc}(p(A \otimes \mathcal K)p).$$
Together with the opposite inequality proved above, one has $$ \mathrm{rc}(p(A \otimes \mathcal K)p) = \sup\{r_\infty(\tau): \tau(p) = 1,\ \tau \in \mathrm{T}^+(A)\},$$ as asserted.
\end{proof}

\begin{rem}
Since $r_\infty$ is upper semicontinuous, one has 
$$\sup\{r_\infty(\tau): \tau(p) = 1,\ \tau \in \mathrm{T}^+(A)\} = \max\{r_\infty(\tau): \tau(p) = 1,\ \tau \in \mathrm{T}^+(A)\}, $$ where $p \in A \otimes \mathcal K$ is a projection.
\end{rem}

\begin{rem}
Does the comparison radius function $r_\infty$ exist for every simple C*-algebra? At least for simple C*-algebras of stable rank one? It also would be interesting to see how the argument of Theorem \ref{low-env-gn} can be adapted to the AF-Villadsen algebras.
%(It is enough to show that $G_A$ is a downward lattice, i.e., if $f, g \in G_A$, then $\min\{f, g\} \in G_A$.) %That is, for any simple C*-algebra $A$, is there an (upper semicontinuous) affine function $r_\infty$ which has Properties (1), (2), (3) of Theorem \ref{low-env-gn}.
\end{rem}

\bibliographystyle{plainurl}
\bibliography{operator_algebras}

\end{document}